\documentclass[12pt,a4paper]{article}
\usepackage[latin1]{inputenc}
\usepackage{amsmath, amsthm, amssymb, amsfonts, amscd}
\usepackage[english]{babel}
\usepackage[all]{xy}
\usepackage[linktocpage=true]{hyperref}
\usepackage{color}
\usepackage[table]{xcolor}
\usepackage{pinlabel}
\usepackage{tikz-cd}
 \usepackage{booktabs}
\usepackage{enumitem}
\usepackage{geometry}
\usepackage{mathtools}
\usepackage{authblk}
\usepackage{thmtools}

\numberwithin{figure}{section}
\numberwithin{equation}{section}
\declaretheoremstyle[bodyfont=\normalfont,spaceabove=\medskipamount,
    spacebelow=\medskipamount]{definition}

\theoremstyle{definition}

\newtheorem{theorem}{Theorem}[section]
\newtheorem{lemma}[theorem]{Lemma}
\newtheorem{corollary}[theorem]{Corollary}
\newtheorem{proposition}[theorem]{Proposition}
\newtheorem{definition}[theorem]{Definition}

\newtheorem{remark}[theorem]{Remark}

\newtheorem{example}[theorem]{Example}

\DeclareMathOperator\Hom{Hom} %Hom set in a category
\DeclareMathOperator\Ad{Ad} % adjoint action on Lie group
\title{\large \textbf{FRAMED KNOTOIDS AND THEIR QUANTUM INVARIANTS}}
\author[$\dagger$]{\normalsize WOUT MOLTMAKER}

\affil[$\dagger$]{\textit{Mathematical Institute, University of Oxford}}

\date{}

\begin{document}

\maketitle

\begin{abstract}
    We modify the definition of spherical knotoids to include a framing, in analogy to framed knots, and define a further modification that includes a secondary `coframing' to obtain `biframed' knotoids. We exhibit topological spaces whose ambient isotopy classes are in one-to-one correspondence with framed and biframed knotoids respectively. We then show how framed and biframed knotoids allow us to generalize quantum knot invariants to a knotoid setting, leading to the construction of general Reshetikhin-Turaev type biframed knotoid invariants.
    % Something like "...using braided groups, introduced by S. Majid."?
\end{abstract}

% \tableofcontents

% \newpage

\section{Introduction and Motivation}

Knotoids were introduced by V. Turaev in \cite{turaev2012knotoids} as a generalization of knots. Intuitively, knotoids are knot diagrams that have open ends; see Figure \ref{fig:planar_knotoids} below. Unlike long knots \cite[~Ch.~1]{chmutov2012introduction}, these end-points are allowed to lie in the interior of the diagram. To prevent all knotoid diagrams from being trivial, the diagrammatic moves in Figure \ref{fig:forbiddenmoves} are explicitly forbidden.
\begin{figure}[ht]
    \centering
    \includegraphics[width=.7\linewidth]{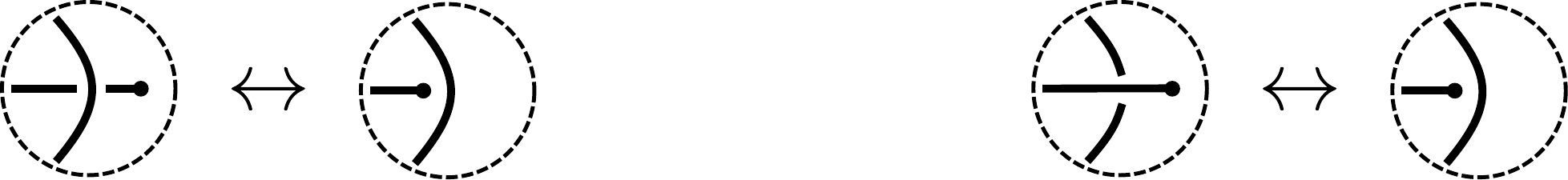}
    \caption{The forbidden knotoid diagram moves.}
    \label{fig:forbiddenmoves}
\end{figure}
Turaev showed that such diagrams are in correspondence with so-called `simple theta-curves', in the same way that knot diagrams correspond to knots. In this sense we say that simple theta-curves are the `geometric realization' of knotoid diagrams.

Knotoids are a natural object to use when studying the knottedness of open-knotted protein chains. This is done in e.g.~\cite{goundaroulis2020knotoids} where the knottedness of a protein is analysed by projecting it onto planes in many possible directions and identifying the resulting knotoids. The result is a map from a discretization of $S^2$ to the set of all knotoids, which encodes much information about the shape of the protein.

To make this application of knotoids viable for more complex proteins, it is helpful to have a classification of knotoids. The classification of knotoids on the sphere with up to 6 crossings is complete \cite{goundaroulis2019systematic}, and makes use of strong knotoid invariants such as the arrow polynomial from \cite{gugumcu2017new} and the double branched cover from \cite{barbensi2018double}. For further progress a larger selection of knotoid invariants is needed. Extending this selection is the primary goal of of this paper.

In this paper the notion of a knotoid diagram is generalized to include a \textit{framing}, in analogy to framed knots. It is generally easier to construct invariants of \textit{framed} knots since framed knot diagrams have fewer allowed Reidemeister moves than knot diagrams. This same principle motivates the introduction of framed knotoids to facilitate the production of knotoid invariants. Along with definitions of several notions of framed knotoid diagrams we will also give geometric realizations of these types of diagrams. This helps to keep the theory grounded in topology, before allowing algebraic and diagrammatic manipulations to take over when we move on to constructing framed knotoid invariants.

This construction of framed knotoid invariants leads to the construction of general Reshetikhin-Turaev type quantum invariants of knotoids. This is related to the very recent work by G\"{u}g\"{u}mc\"{u} and Kauffman \cite{gugumcu2021quantum} on quantum knotoid invariants. Our work in this paper is independent from theirs, but there are similarities between our approach and theirs: e.g.~what we shall call our `coframing' is analogous to their `rotation number'. Our approach will add some topological motivation to the ideas developed there, and as a result has a focus on spherical knotoids rather than planar ones.\\

\section{Definitions and Geometric Realizations}\label{sec:topology}

In this section we briefly recall the basic results on knotoids and simple theta-curves, both due to Turaev \cite{turaev2012knotoids}. Afterwards we introduce framed and biframed versions of both knotoids and simple theta-curves.

\subsection{Knotoids and Simple Theta-curves}

\begin{definition}\cite{turaev2012knotoids}
Let $\Sigma$ be a surface. A \textbf{knotoid diagram} in $\Sigma$ is a smooth immersion $\phi:[0,1]\hookrightarrow \Sigma^\circ$ into the interior of $\Sigma$, whose only singularities are transversal double points endowed with over/undercrossing data. Knotoid diagrams have a natural orientation, namely from the \textbf{leg} $\phi(0)$ to the \textbf{head} $\phi(1)$. We will not always label the leg and head or depict knotoids as oriented, but this orientation is always implicit. Two knotoid diagrams are \textbf{equivalent} if they can be related by a sequence of isotopies of $\Sigma$ and applications of the \textit{Reidemeister moves} $R1,R2,R3$ familiar for knot diagrams. \textbf{Knotoids} are the equivalence classes of knotoid diagrams.

Applications of the Reidemeister moves are not allowed to involve the end-points, and indeed the diagrammatic moves in Figure \ref{fig:forbiddenmoves} are expressly forbidden.
\end{definition}

\begin{figure}[ht]
    \centering
    \includegraphics[width=.95\linewidth]{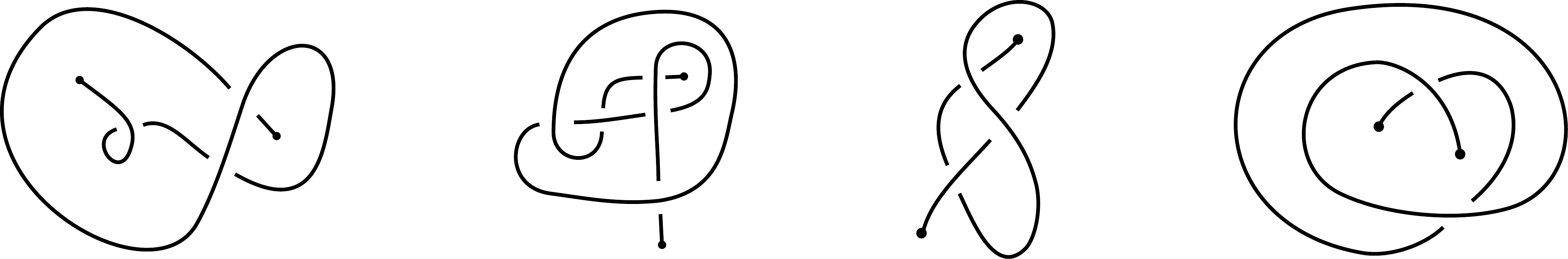}
    \caption{Examples of knotoid diagrams on $\Sigma=\mathbb{R}^2$.}
    \label{fig:planar_knotoids}
\end{figure}

The following is immediate from noting that $S^2$ is isomorphic to the one-point compactification of $\mathbb{R}^2$.

\begin{remark}\label{rk:spherical}
Knotoids in $\Sigma=S^2$ are equivalent to knotoids in $\Sigma=\mathbb{R}^2$ modulo the \textbf{spherical move} $R4$ depicted in Figure \ref{fig:sphericalmove}.
\begin{figure}[ht]
    \centering
    \includegraphics[width=.35\linewidth]{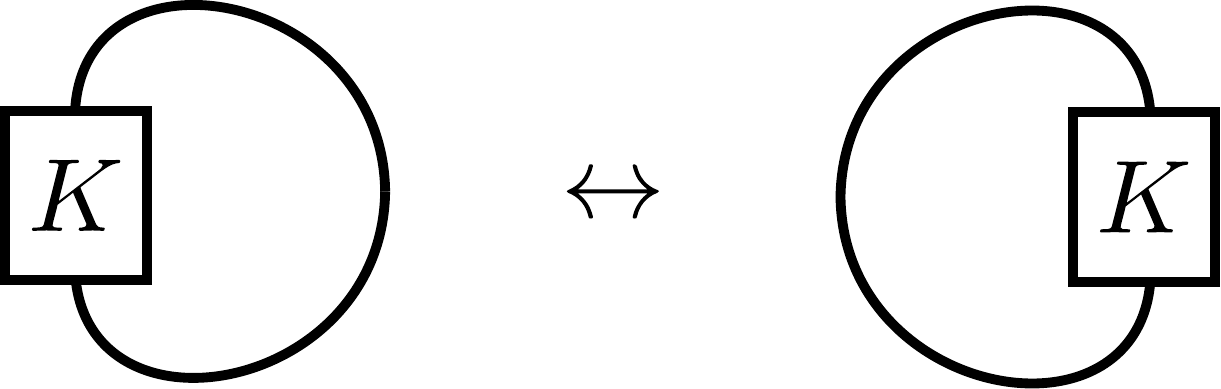}
    \caption{The spherical move on a knotoid diagram, where $K$ denotes the rest of a fixed knotoid diagram.}
    \label{fig:sphericalmove}
\end{figure}
Note that while this move needs to be imposed on knotoid diagrams in $\mathbb{R}^2$ and may change the equivalence class of a knotoid diagram in $\mathbb{R}^2$, it is a surface isotopy for knotoids in $S^2$ given by sweeping a strand external to the diagram along the back of $S^2$. Nevertheless, it will be helpful to distinguish this move from other isotopies, so that we are free to think of knotoids in $S^2$ as lying in $\mathbb{R}^2$.
% Note that while this is a diagrammatic move of knotoids on $\mathbb{R}^2$, it is an ambient isotopy for knotoids on $S^2$ given by sweeping a strand external to the diagram along the back of $S^2$. Nevertheless, it will be helpful to distinguish this move from other isotopies, so that we are free to think of knotoids in $S^2$ as lying in $\mathbb{R}^2$.
\end{remark}

We call a knotoid in $S^2$ \textbf{spherical}. In this paper all knotoids are assumed to be spherical unless stated otherwise.

\begin{remark}
In the remainder of this section we will consider several kinds of embeddings of topological spaces in $S^3$, related to `theta-curves' defined below. All of these topological spaces will be locally homeomorphic to $\mathbb{R}$, $\mathbb{R}^2$, or $\mathbb{R}\times[0,\infty)$ except at two points or arcs. All the embeddings that we will consider in what follows will be implicitly assumed to be smooth everywhere except for these points or arcs, where smoothness is not defined.
\end{remark}

\begin{definition}\label{def:thetacurve}
\cite{turaev2012knotoids}
A \textbf{theta-curve} $\theta$ is an embedding into $S^3$ of the graph $\Theta$, which consists of two vertices $v_0,v_1$ and three edges $e_-,e_0,e_+$ joining them. We consider such embeddings up to ambient isotopies of $S^3$ that preserve the labels of the vertices and edges. If $\theta$ is a theta-curve, then omitting any edge from $\Theta$ results in an embedding $S^1\hookrightarrow S^3$, i.e.~a knot. A theta-curve $\theta$ is said to be \textbf{simple} if the image of $e_-\cup e_+$ is the unknot.
\end{definition}

The reason for introducing simple theta-curves is the following geometric realization result, also due to Turaev \cite{turaev2012knotoids}:

\begin{theorem}\label{thm:knotoidbijection}
There is a bijection between spherical knotoids and label-preserving ambient isotopy classes of simple theta-curves. We say that $\Theta$ is therefore the \textbf{geometric realization} of knotoids.
\end{theorem}

For a detailed proof, see \cite[~Sec.~6]{turaev2012knotoids}. For us the important parts of the proof are the explicit correspondence, and the notion of a \textit{standard} theta-curve. Intuitively a theta-curve is standard if $e_+\cup e_-$ is manifestly the unknot:

\begin{definition}
A simple theta-curve $\theta$ is \textbf{standard} if the vertices of $\theta$ lie in $\mathbb{R}^2\times\{0\}\subseteq S^3\cong \mathbb{R}^3\cup\{\infty\}$, $\theta(e_+),\theta(e_-)$ lie in the upper- and lower half-plane of $\mathbb{R}^2\times \{0\}$ respectively, and $\theta(e_+),\theta(e_-)$ project bijectively to the same arc $a\subseteq\mathbb{R}^2\times\{0\}$ that connects the vertices of $\theta$.
\end{definition}

Turaev shows that any simple theta-curve is ambient isotopic to a standard one, and that if $\theta,\theta'$ are ambient isotopic standard simple theta-curves then they are ambient isotopic within the class of standard simple theta-curves \cite{turaev2012knotoids}. For standard simple theta-curves the bijection of Theorem \ref{thm:knotoidbijection} is depicted in Figure \ref{fig:bijection}.

\begin{figure}[ht]
    \centering
    \includegraphics[width=.7\linewidth]{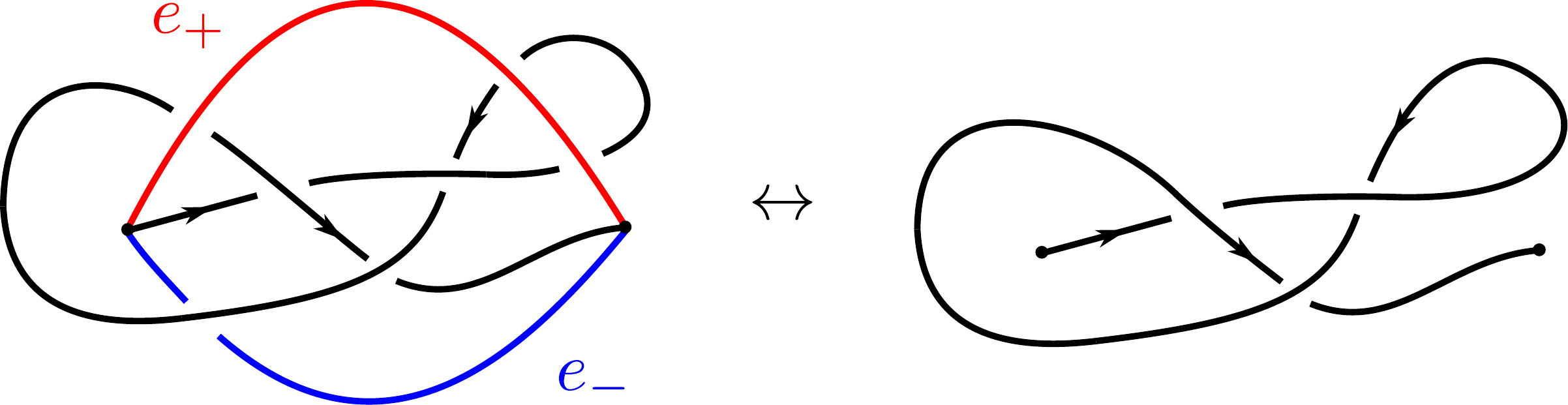}
    \caption{The bijection of Theorem \ref{thm:knotoidbijection}. A standard theta-curve produces a knotoid by moving $\theta(e_0)$ onto a neighborhood of $\mathbb{R}^2\times\{0\}$ and taking a vertical planar projection. A knotoid conversely produces a theta-curve by considering a knotoid diagram as a knotted arc lying in a neighbourhood of $\mathbb{R}^2\times \{0\}$ and tying arcs $e_\pm$ to its end-points.}
    \label{fig:bijection}
\end{figure}

\subsection{Framed and Biframed Knotoids}\label{subsec:framedandbiframed}

The aim of this section is to give reasonable definitions of what it means for a knotoid to be \textit{framed}. First recall the main results on framed knots from \cite{elhamdadi2020framed}: A framed knot $K$ is a knot with a transversal, smooth, everywhere nonzero vector field. The framing of $K$ is then defined to be the associated element in $\pi_1(SO(2))\cong\mathbb{Z}$. There is an equivalence between framed knots and tangled \textit{ribbons} in $S^3$, so framed knots can equivalently be characterized as embeddings of the annulus into $S^3$. 

Any knot diagram induces a \textit{blackboard framing} on the corresponding knot, with framing integer given by the \textit{writhe} of the diagram. Letting $R1'$ denote the \textit{weakened} first Reidemeister move (see Figure \ref{fig:weakreid}), framed knots up to equivalence are in bijection with knot diagrams up to $R1',R2,R3$. Such knot diagrams are referred to as \textit{framed} knot diagrams. Details can be found in \cite{elhamdadi2020framed}.
\begin{figure}[ht]
    \centering
    \includegraphics[width=.3\linewidth]{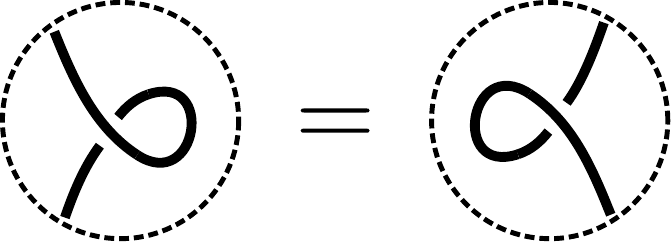}
    \caption{The weakened first Reidemeister move, $R1'$.}
    \label{fig:weakreid}
\end{figure}

\subsubsection{Framed Knotoids}\label{subsec:framedknotoids}

Given that framed knot diagrams are simply knot diagrams up to the moves $R1',R2,R3$, the definition of framed knotoids is natural:

\begin{definition}
A \textbf{framed knotoid} is an equivalence class of knotoid diagrams under the equivalence generated by ambient isotopies and $R1',R2,R3$. A diagram representing a framed knotoid is also called a \textbf{framed knotoid diagram}.
\end{definition}

As the terminology suggests, framed knotoid diagrams are just knotoid diagrams with a `framing integer' attached. Our immediate goal is to prove this. To extract an integer from a given diagram $K$ in a canonical way we proceed in analogy with framed knots:

\begin{definition}\label{def:writhe}
Let $K$ be a framed knotoid diagram. We define the \textbf{blackboard framing} of $K$ to be the writhe of $K$, i.e.~the sum of the signs of all crossings of $K$. Here the sign of a crossing is defined analogously as for knot diagrams, see \cite{elhamdadi2020framed}.
\end{definition}

The blackboard framing is an integer that is also well-defined for the framed knotoid that $K$ represents. This last statement follows from invariance of the writhe under $R1',R2,R3$. To justify the use of the term `blackboard framing' for the writhe of a knotoid in analogy with framed knots, we need to interpret the framed knotoid as an object in $S^3$ with an associated canonical transversal vector field.
% To realize the writhe of a knotoid as a genuine form of blackboard framing, we need to interpret the framed knotoid as an object in $S^3$ with an associated canonical transversal vector field. 
We do so after introducing framed theta-curves in the next subsection. For now it suffices to work with the combinatorially defined writhe, without reference to a geometric realization. Using the writhe, we conclude the following lemma:

\begin{lemma}\label{lm:framedknotoidbijection}
There is a bijection
\begin{align}
    \{ \text{Framed knotoids} \} &\leftrightarrow \{ \text{Knotoids} \}\times\mathbb{Z} \label{eq:framedknotoidbijection}\\
    K &\mapsto \left(K,\text{writhe}(K)\right) \nonumber.
\end{align}
\end{lemma}

\begin{proof}
First note that surjectivity is clear. Indeed, given a knotoid diagram $K$ we can add $R1$ loops to adjust the framing of $K$ at will without altering the underlying knotoid. (By an $R1$ loop we mean a loop like those pictured in Figure \ref{fig:weakreid}.) For injectivity, suppose that $K,K'$ are equivalent knotoid diagrams with equal writhe. Then we must show that $K,K'$ are equivalent as framed knotoid diagrams. To see this, first note that positive and negative $R1$ loops can be cancelled against each other via a combination of $R2$ and $R3$ moves. See Figure \ref{fig:cancellingloops}.

\begin{figure}[ht]
    \centering
    \includegraphics[width=.9\linewidth]{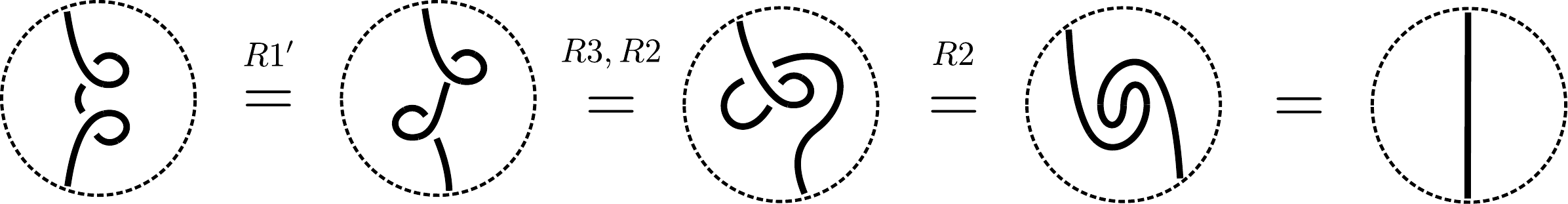}
    \caption{Cancellation of opposite $R1$ loops.}
    \label{fig:cancellingloops}
\end{figure}

Since $K,K'$ are equivalent as knotoids there is a sequence of isotopies of $S^2$ and Reidemeister moves that turns $K'$ into $K$. By the above, we can modify this sequence of moves into a sequence of framed knotoid diagram moves. Namely we replace every deletion of an $R1$ loop by the sequence of isotopies of $S^2$ and $R2,R3$ moves that brings this $R1$ loop to a neighbourhood of the head of $K'$, and every creation of an $R1$ loop by the creation of that loop and its opposite, followed by moving this opposite to the head of $K'$. The result is the diagram $K$, but with a finite number of $R1$ loops concentrated in a neighbourhood of the head. Since $K'$ and $K$ have equal writhe, the sum of the signs of these loops is $0$. Thus we can apply the moves in Figure \ref{fig:cancellingloops} finitely many times to remove all these loops, ending up with a diagram of $K$.
\end{proof}

% Finally, maybe a word on the redundancy of R1' for knot diagrams on S^2, and a look at whether this generalizes to knotoids. (I don't think it does)

\subsubsection{Framed Simple Theta-curves}\label{subsec:framedtheta}

Our next goal is to find a geometric realization of framed knotoids, and to interpret their writhe in terms of the framing of this geometric realization.

This means that we want topological spaces whose ambient isotopy classes in $S^3$ are in bijection with framed knotoids. Recall that framed knots are equivalent to tangled ribbons in $S^3$, and hence the geometric realization of framed knot diagrams is the annulus. Also recall the sketch proof of Theorem \ref{thm:knotoidbijection}, which implies that knotoids correspond to the edge $e_0$ of a simple theta-curve. From these two facts it is reasonable to expect that framed knotoids correspond to simple theta-curves in $S^3$, with $e_0$ thickened into a ribbon. The associated geometric realization is pictured in Figure \ref{fig:fancy_framed_thetacurve}.

\begin{figure}[ht]
    \centering
    \includegraphics[width=.6\linewidth]{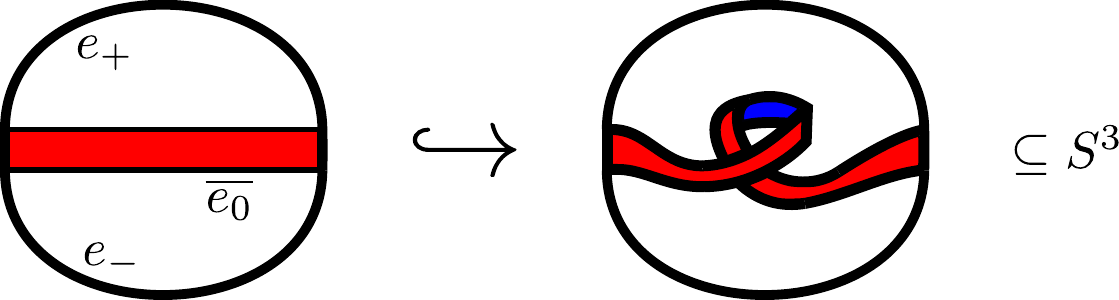}
    \caption{A simple framed theta-curve.}
    \label{fig:fancy_framed_thetacurve}
\end{figure}

\begin{definition}
Let $\overline{\Theta}$ denote the space $\Theta$ with the middle edge $e_0$ thickened into a 2-dimensional ribbon $\overline{e}_0$ attached along neighbourhoods of $v_0,v_1$ as depicted in Figure \ref{fig:fancy_framed_thetacurve}. A \textbf{framed theta-curve} is an embedding $\overline{\Theta}\hookrightarrow S^3$. We consider such embeddings up to label-preserving ambient isotopies of $S^3$.

We say that a framed theta-curve is \textbf{simple} if it has integer framing and the image of $e_+\cup e_-$ is the unknot, the former meaning that if we put $\theta$ in a form such that $e_+\cup e_-$ is manifestly unknotted then the same side of $\overline{e}_0$ must face `outwards' at $v_0$ and $v_1$. In other words, it means that we want $\theta(\overline{e}_0)$ to make an integer number of turns through $S^3$, in some sense; see Figure \ref{fig:fancy_framed_thetacurve}. 
\end{definition}

To make this notion of `turns' more precise we define standard framed simple theta-curves and their framing integers:

\begin{definition}
Let $\mathbb{R}^2$ denote the plane $\mathbb{R}^2\times \{0\} \subseteq \mathbb{R}^3\subseteq S^3$. In analogy to the proof of Theorem \ref{thm:knotoidbijection} we say that a framed simple theta-curve $\theta$ is \textbf{standard} if the same assumptions on $\theta(e_+),\theta(e_-)$ as for standard theta-curves hold, \textit{and} the neighbourhoods of $\theta(e_+\cup e_-)$ where $\theta(\overline{e}_0)$ is attached are parallel straight lines in $\mathbb{R}^2$. In particular the example in Figure \ref{fig:fancy_framed_thetacurve} is standard.

Let $e_0$ be one of the lengths of the boundary of $\theta(\overline{e}_0)$. The direction into the interior of $\overline{e}_0$ that is transversal to $e_0$ defines a path in $SO(2)$ when considered along all of $\theta(\overline{e}_0)$. In the case of a \textit{standard} simple framed theta-curve this path is a loop by assumption. The corresponding element of $\pi_1(SO(2))\cong \mathbb{Z}$ is defined to be the \textbf{framing integer} $\text{fr}(\theta)$ of $\theta$.
\end{definition}

\begin{remark}\label{rk:framing_extension}
Note that any framed simple theta-curve can be brought into standard form. Indeed, it suffices to extend the same moves that bring a simple theta-curve into standard form to a framed simple theta-curve, and then to adjust regular neighbourhoods of the end-points to be vertical. (For the definition of a regular neighbourhood of an end-point, see \cite{turaev2012knotoids}. Here we assumed without loss of generality that the attachments of $\overline{e}_0$ lie inside such a neighbourhood.) This works because any move on theta-curves can easily be extended to a move of framed theta-curves, up to deformations that shrink a portion of the ribbon $\overline{e}_0$. We will see in subsection \ref{subsec:biframed} that for biframed theta-curves this is not quite the case.

It is clear that the framing integer is independent of the standard form produced in this way. Indeed, the \textit{relative} twisting between the transversal vectors to $e_0$ into $\overline{e}_0$ at $v_0,v_1$ when transforming a theta-curve into standard form is clearly independent of the standard form chosen. 
\end{remark}

Remark \ref{rk:framing_extension} allows us to define the framing integer of \textit{any} simple theta-curve $\theta$, namely as the framing integer of any standard theta-curve equivalent to $\theta$. Using this we can conclude the following:

\begin{lemma}\label{lm:framedthetabijection}
There is a bijection
\begin{align}
    \{ \text{Framed simple theta-curves} \} &\leftrightarrow \{ \text{Simple theta-curves} \}\times\mathbb{Z} \label{eq:framedthetabijection}\\
    \theta &\mapsto \left(\theta^\circ,\text{fr}(\theta)\right), \nonumber
\end{align}
where $\theta^\circ$ is the simple theta-curve obtained from the projection $\overline{e}_0\twoheadrightarrow e_0$ of a ribbon onto one of its $1$-dimensional boundary lengths (which is a smooth arc).
\end{lemma}

\begin{proof}
This proof is analogous to that of Lemma \ref{lm:framedknotoidbijection}: surjectivity is clear. For injectivity suppose that $\theta_1,\theta_2$ have equal framing and that $\theta_1^\circ\simeq\theta_2^\circ$. Then an ambient isotopy relating $\theta_1^\circ$ to $\theta_2^\circ$ can be extended to an ambient isotopy relating $\theta_1$ to a theta-curve $\theta_3$ such that $\theta_3^\circ$ is \textit{equal} to $\theta_2^\circ$. This extension may end up adding extra twists (both clockwise and counter-clockwise) to $\theta_3$ that are not present in $\theta_2$. But since $\text{fr}(\theta_3)=\text{fr}(\theta_1)=\text{fr}(\theta_2)$ these extra twists must cancel to give $\theta_1\simeq \theta_3\simeq \theta_2$, proving injectivity.
\end{proof}

Combining bijection \eqref{eq:framedthetabijection} with bijection \eqref{eq:framedknotoidbijection} and Theorem \ref{thm:knotoidbijection} gives a sequence of bijections

\begin{align}
    \{\text{Framed knotoids}\} &\xleftrightarrow{\text{writhe}} \{\text{Knotoids}\}\times \mathbb{Z} \nonumber\\
    &\xleftrightarrow{\text{Thm } \ref{thm:knotoidbijection}} \{\text{Simple theta-curves}\}\times \mathbb{Z} \nonumber\\
    &\xleftrightarrow{\text{framing}} \{\text{Framed simple theta-curves}\}. \label{eq:framedrealization}
\end{align}

Thus we obtain the following corollary to Theorem \ref{thm:knotoidbijection}:

\begin{corollary}\label{cor:framedrealization}
The geometric realization of framed knotoids is $\overline{\Theta}$, i.e.~the equivalence classes of embeddings of $\overline{\Theta}$ are in bijective correspondence with framed knotoids.
\end{corollary}

\begin{remark}\label{rk:bijection}
A priori the bijection of Corollary \ref{cor:framedrealization} is not topological in nature: it simply says that both framed knotoids and framed simple theta-curves correspond to knotoids with an integer attached. As with framed knots, the topological interpretation of Corollary \ref{cor:framedrealization} passes through the blackboard framing.

\begin{figure}[ht]
    \centering
    \includegraphics[width=.25\linewidth]{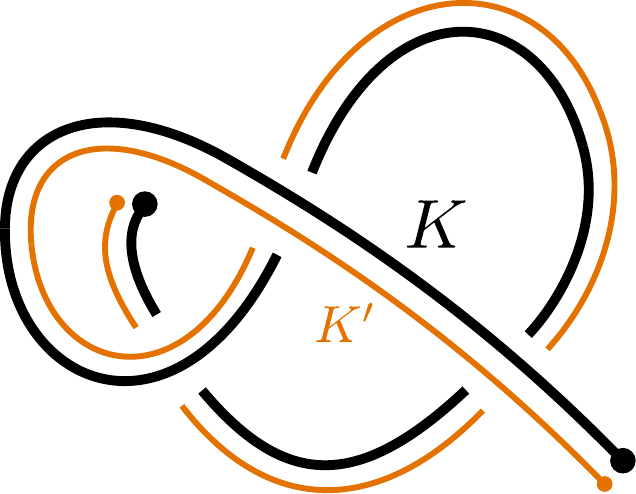}
    \caption{The blackboard framing of a knotoid diagram.}
    \label{fig:blackboard}
\end{figure}

Given a knotoid diagram $K$, let $K'$ be given by a small shift transversal to $K$; see Figure \ref{fig:blackboard}. Under the inverse of the projection assignment from Theorem \ref{thm:knotoidbijection}, this $K'$ specifies a transversal vector field on $e_0$ of the standard theta-curve corresponding to $K$. (As a convention, we take this vector field to be \textit{upwards} in $\mathbb{R}^2\times\{0\}$ at $v_0,v_1$.) In turn this specifies a ribbon extension $\overline{e}_0$ to $e_0$. The corresponding framing is defined to be the \textit{blackboard framing} of $K$. This is a number equal to the writhe of $K$. The proof of this is analogous to that for framed knots; see \cite{elhamdadi2020framed}.

In this way, the notions of framing for knotoids and theta-curves coincide under the map $\{\text{Knotoids}\}\to\{\text{Simple theta-curves}\}$ from Theorem \ref{thm:knotoidbijection}.
\end{remark}

Framed simple theta-curves, and hence framed knotoids, could be useful in their own right for modelling DNA structures. Namely the framing can be used to model the double helix structure of a DNA strand more accurately. A consequence is that framed knotoid models of DNA can detect features such as \textit{supercoiling} \cite{bates2005dna}, which is undetectable by knotoid models due to the $R1$ relation. For such modelling purposes, it is natural to ask whether knotoids can be generalized to encode non-integer framing. Clearly this can be done for framed simple theta-curves, but there is no simple way to translate this to knotoid diagrams.

Defining a notion of half-integer framing that is easy to recognize in diagrams is one of the motivations for introducing `\textit{biframed}' knotoids and theta-curves in the next subsections.

% For such modelling purposes, it is natural to ask whether knotoids can be generalized to encode non-integer framing. This is one of our motivations for introducing \textit{biframed knotoids} in the next subsections.

\subsubsection{Biframed Simple Theta-curves}\label{subsec:biframed}

To define biframed knotoids we will work backwards; we start by defining the geometric realization and its embeddings, and consider the associated diagrams in the next subsection. The idea of a biframed theta-curve is to take a framed theta-curve, and add a secondary \textit{coframing} by allowing the outer ring $e_+\cup e_-$ to carry framing information. The associated geometric realization is depicted in Figure \ref{fig:fancy_biframed_thetacurve}.

\begin{figure}[ht]
    \centering
    \includegraphics[width=.6\linewidth]{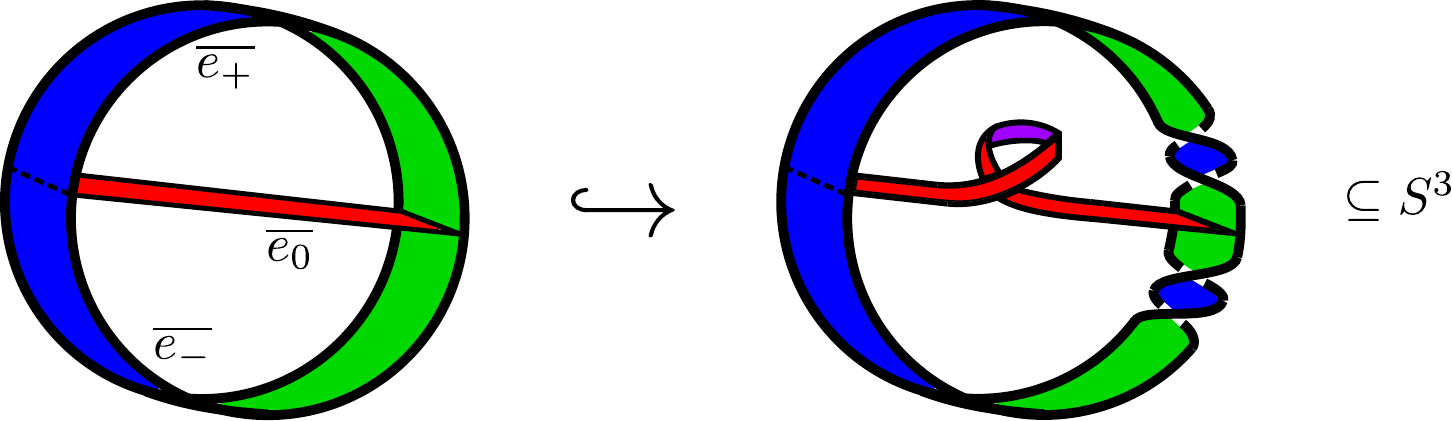}
    \caption{A (simple) biframed theta-curve.}
    \label{fig:fancy_biframed_thetacurve}
\end{figure}

\begin{definition}
Let $\lvert\Theta\rvert$ denote the space depicted on the left in Figure \ref{fig:fancy_biframed_thetacurve}, i.e.~the space $\Theta$ with $\{e_-,e_0,e_+\}$ thickened into ribbons $\{\overline{e}_-,\overline{e}_0,\overline{e}_+\}$ and with $\overline{e}_0$ attached at attaching arcs $a_0,a_1$ along the width of $\overline{e}_-\cup\overline{e}_+$. This means that if we note that $\overline{e}_+\cup\overline{e}_- \cong S^1\times [0,1]$ then the arcs $a_i$ are equal to $p_i\times [0,1]$ for some $p_i\in S^1$ and $i\in\{0,1\}$. A \textbf{biframed theta-curve} is an embedding $\theta:\lvert\Theta\rvert\hookrightarrow S^3$. We consider such embeddings up to label-preserving ambient isotopies of $S^3$.
\end{definition}

To justify the terminology `biframed', we give a way of assigning \textit{two} numbers to (a certain class of) biframed theta-curves. One of these is analogous to the framing of a framed theta-curve. The other is derived from the fact that for a biframed theta-curve $\overline{e}_-\cup\overline{e}_+$ is also a ribbon, and shall be referred to as the `coframing' of a biframed theta-curve. As for framed theta-curves, to facilitate the definition of these numbers we first define a class of `standard' theta-curves.

\begin{definition}\label{def:standardbi}
A biframed theta-curve $\theta$ is said to be \textbf{standard} if the following conditions hold:
\begin{itemize}
    \item The annulus $\theta(\overline{e}_-\cup\overline{e}_+)$ is the unframed (i.e.~framing $0$) unknot, and $\theta(a_i)$ lies in the horizontal plane $\mathbb{R}^2\times\{0\}\subseteq S^3$ for $i\in\{0,1\}$.
    \item The annulus $\theta(\overline{e}_-\cup\overline{e}_+)$ is in manifestly unframed form, i.e.~of the form $S^1\times[-\epsilon,\epsilon]$ with $S^1\subseteq\mathbb{R}^2\times\{0\}\subseteq S^3$. 
    % (See the left- and right-most theta-curves in Figure \ref{fig:coframing_example}.)
    \item Let $e_0$ denote one of the lengths of the boundary of $\theta(\overline{e}_0)$. Some neighbourhoods of the attaching arcs of $\theta(\overline{e}_0)$ to $\theta(\overline{e}_-\cup\overline{e}_+)$ lie in a single plane $P$, and the tangent of $e_0$ is perpendicular to $\theta(\overline{e}_-\cup\overline{e}_+)$ at both attaching points.
    \item Away from the attaching arcs, $\theta(\overline{e}_0)$ does not intersect the vertical bars perpendicular to $P$ that are depicted in Figure \ref{fig:bars}.
\end{itemize}

\begin{figure}[ht]
    \centering
    \includegraphics[width=.22\linewidth]{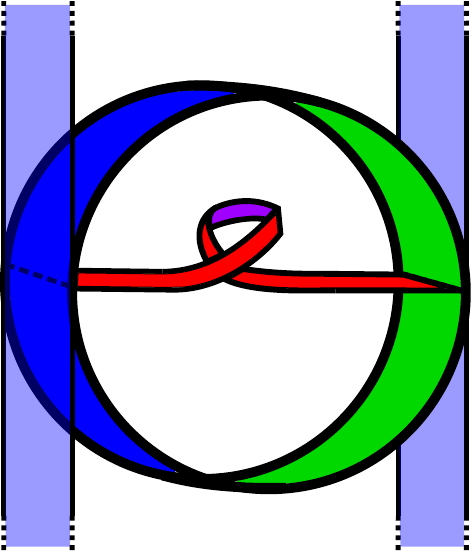}
    \caption{A standard embedding of $\lvert\Theta\rvert$, along with the vertical bars from Definition \ref{def:standardbi} that $\theta(\overline{e}_0)$ is not allowed to intersect.}
    \label{fig:bars}
\end{figure}
\end{definition}

\begin{definition}\label{def:coframing}
Let $\theta,e_0$ be as in Definition \ref{def:standardbi}. As with framed theta-curves, the direction of the interior of $\theta(\overline{e}_0)$ perpendicular to $e_0$ determines a framing integer in $\pi_1(SO(2))\cong\mathbb{Z}$. To define the coframing, parametrize $e_0$ by $t\in[0,1]$ such that $e_0(0)\in a_0$ and $e_0(1)\in a_1$. Note that $e_0$ is a smooth arc by assumption on $\theta: \lvert\Theta\rvert\hookrightarrow S^3$. For $t\in(0,1)$ and $i\in\{0,1\}$, consider the vectors in $S^3$ from $e_0(i)$ to $e_0(t)$. Project these vectors onto $P$ and normalize them. This is possible by the fourth assumption of standard-ness: note that without this assumption some of these projections onto $P$ could yield the zero vector. This results in paths $p_i$ in $SO(2)$, which are loops by the first and third assumption of standard-ness. Let $n_i$ denote the associated elements of $\pi_1(SO(2))\cong\mathbb{Z}$. Then we define the \textbf{coframing} of $\theta$ to be given by
\begin{equation}\label{eq:coframing}
    \text{cofr}(\theta) = n_0 - n_1.
\end{equation}
\end{definition}

It is easy to see that any biframed theta-curve $\theta$ for which $\theta(\overline{e}_-\cup\overline{e}_+)$ is the unframed unknot can be put into standard form: moves of a theta-curve \textit{away from the attaching points} extend readily to moves of biframed theta-curves. Moves at the attaching points can be extended by replacing swivels around the end-points by twists that drag $\overline{e}_0$ along, as in the right-hand side of Figure \ref{fig:coframing_example}. After such an extension, $\theta$ can be put into standard form by applying moves such as that depicted in the right-hand side of Figure \ref{fig:coframing_example} to put $\theta(\overline{e}_-\cup\overline{e}_+)$ into standard unframed form, followed by deformations in regular neighbourhoods of $a_0,a_1$ and deformations moving $\theta(\overline{e}_0)$ away from the bars of Figure \ref{fig:bars}.

\begin{remark}\label{rk:coframing_extension}
As before it is clear that $\text{fr}(\theta)$ is independent of how $\theta$ is put into standard form, but this is not immediate for $\text{cofr}(\theta)$: the ambiguity arises because we can undo a twist in $\theta(\overline{e}_-\cup\overline{e}_+)$ either at $a_0$ or at $a_1$, dragging $\overline{e}_0$ around one end of $\theta(\overline{e}_-\cup\overline{e}_+)$ or the other. This is illustrated in Figure \ref{fig:coframing_welldefinedness}.

\begin{figure}[ht]
    \centering
    \includegraphics[width=\linewidth]{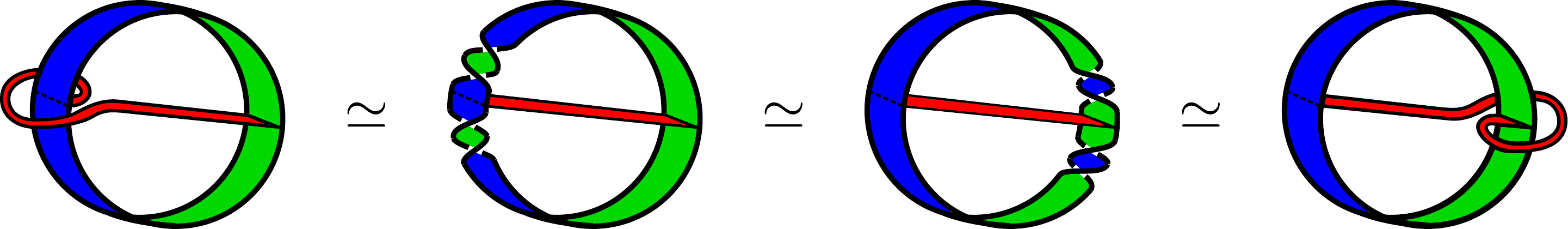}
    \caption{Two ways to bring a $(0,1)$-biframed theta-curve into standard form.}
    \label{fig:coframing_welldefinedness}
\end{figure}

As can be read off immediately from Figure \ref{fig:coframing_welldefinedness}, the difference between undoing a twist in $\theta(\overline{e}_-\cup\overline{e}_+)$ at one end or the other is a difference $n_i\to n_i-1$ for $i\in\{0,1\}$. The minus sign in Equation \eqref{eq:coframing} cancels this difference, so that the coframing is independent of the chosen standard form. Analogous reasoning holds for oppositely directed twist in $\theta(\overline{e}_-\cup\overline{e}_+)$. Thus Definition \ref{def:coframing} gives a well-defined coframing for \textit{any} biframed theta-curve $\theta$ such that $\theta(\overline{e}_-\cup\overline{e}_+)$ is the unframed unknot.
\end{remark}

Definition \ref{def:coframing} and Remark \ref{rk:coframing_extension} allow us to define the framing and coframing of \textit{any} biframed theta-curve $\theta$ for which $\theta(\overline{e}_-\cup\overline{e}_+)$ is the unframed unknot. Using this we can define the objects that are of interest to us, namely `simple' biframed theta-curves:

\begin{definition}
A biframed theta-curve $\theta$ is \textbf{simple} if the annulus $\theta(\overline{e}_-\cup\overline{e}_+)$ is the unframed unknot and $\theta$ has integer framing and coframing numbers.
\end{definition}

In particular the example shown in Figure \ref{fig:fancy_biframed_thetacurve} is simple. 

\begin{remark}
For both standard and simple biframed theta-curves, note that we have restricted to embeddings of $\lvert\Theta\rvert$ such that $\overline{e}_-\cup\overline{e}_+$ has framing $0$. One may think that the complexity added by making $\overline{e}_-\cup\overline{e}_+$ a ribbon is negated by this restriction. To see that this is not the case, consider Figure \ref{fig:coframing_example}.

\begin{figure}[ht]
    \centering
    \includegraphics[width=\linewidth]{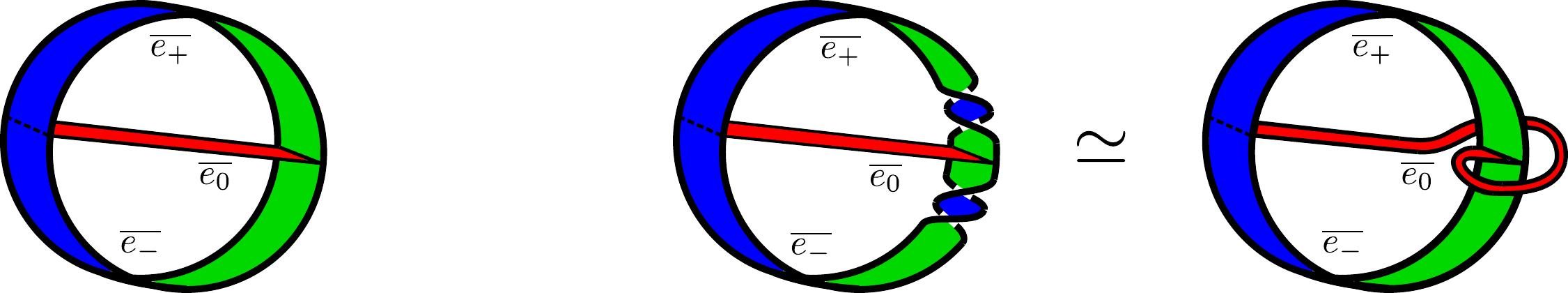}
    \caption{Two standard biframed theta-curves differing only in coframing.}
    \label{fig:coframing_example}
\end{figure}

Figure \ref{fig:coframing_example} depicts two \textit{simple} biframed theta-curves. Indeed, for the second theta-curve $\overline{e}_-\cup\overline{e}_+$ can clearly be deformed into the standard unframed form (as depicted). However, the deformation that accomplishes this drags $\overline{e}_0$ around $\overline{e}_-\cup\overline{e}_+$, as pictured, resulting in a biframed theta-curve that \textit{differs} from the first. The coframing is defined to record exactly such differences.
\end{remark}

% Note that the two biframed theta-curves in Figure \ref{fig:coframing_example} \textit{would} be equivalent as framed theta-curves, since the move relating them would be a swivel of $\overline{e}_0$ around $v_1$, but this move does not extend to a valid move in the biframed case. This small amount of added complexity is referred to as the coframing of the simple biframed theta-curve.

The upshot of the biframing is the following lemma:

\begin{lemma}\label{lm:biframedthetabijection}
There is a bijection:
\begin{align}
    \{ \text{Simple biframed theta-curves} \} &\leftrightarrow \{ \text{Simple theta-curves} \}\times\mathbb{Z}^2 \label{eq:biframedthetabijection}\\
    \theta &\mapsto \left(\theta^\circ,\text{fr}(\theta),\text{cofr}(\theta)\right). \nonumber
\end{align}
Here $\theta^\circ$ is the theta-curve given by any one component of the $1$-dimensional boundary of $\theta(\lvert\Theta\rvert)$.
\end{lemma}

\begin{proof}
This proof is analogous to that of Lemma \ref{lm:framedthetabijection}. In short: surjectivity is clear and for injectivity an equivalence $\theta_1^\circ \cong \theta_2^\circ$ extends to $\theta_1\cong \theta_2$ up to framing twists and coframing loops that must cancel if $\text{fr}(\theta_1)=\text{fr}(\theta_2)$ and $\text{cofr}(\theta_1)=\text{cofr}(\theta_2)$ (e.g.~via Figure \ref{fig:coframing_welldefinedness}).
\end{proof}

% Proving this bijection is analogous to the proof of bijection \eqref{eq:framedthetabijection}. In short: surjectivity is clear and for injectivity an equivalence $\theta_1^\circ \cong \theta_2^\circ$ extends to $\theta_1\cong \theta_2$ up to framing twists and coframing loops that must cancel if $\text{fr}(\theta_1)=\text{fr}(\theta_2)$ and $\text{cofr}(\theta_1)=\text{cofr}(\theta_2)$ (e.g.~via Figure \ref{fig:coframing_welldefinedness}).

\subsubsection{Biframed Knotoids}

Next we define biframed knotoids. As the terminology suggests, this will turn out to be the class of diagrams whose geometric realization is $\lvert\Theta\rvert$. In light of Remark \ref{rk:bijection} and the proof of Theorem \ref{thm:knotoidbijection}, it is no surprise that these are the knotoid diagrams obtained from projecting the smooth arc $e_0$ of a standard biframed theta-curve onto its plane $P$.

\begin{definition}
Pick two points $v_0,v_1\in S^2$. Denote the directed straight arc in $\mathbb{R}^2\cup\{\infty\}\cong S^2$ from $v_0$ to $v_1$ by $L$. A \textbf{biframed knotoid} diagram is a knotoid diagram $K$ in $S^2$ with leg and head equal to $v_0$ and $v_1$ respectively, such that the tangents of $K$ coincide with $L$ in neighbourhoods of $v_0,v_1$. Two biframed knotoid diagrams are said to be \textbf{equivalent} if they can be related by applications of $R1',R2,R3,R4$, ambient isotopies of $S^2$ \textit{away from $v_0,v_1$}, and the \textit{coframing identities} depicted in Figure \ref{fig:coframing_identities} which encode the move in Figure \ref{fig:coframing_welldefinedness}.
\end{definition}
% Note that ambient isotopies away from the end-points can always be constructed from ambient isotopies, by combining them with smooth bump functions around the end-points. (Noting that S^2 is a real manifold.)

\begin{figure}[ht]
    \centering
    \includegraphics[width=\linewidth]{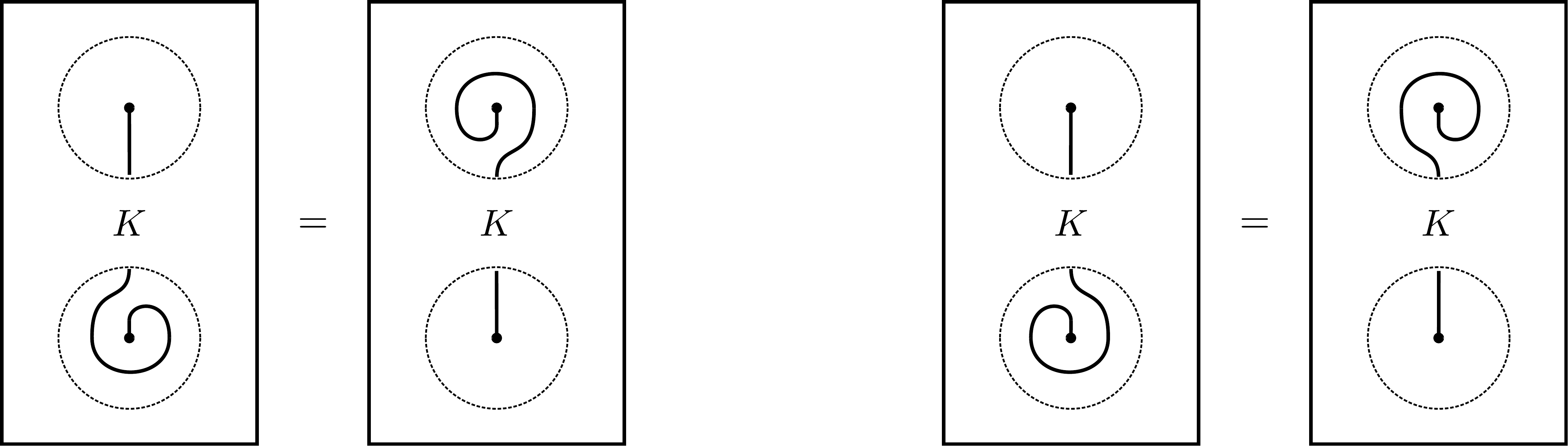}
    \caption{The coframing identities. Here $K$ denotes the rest of the knotoid diagram, which is fixed under the identities.}
    \label{fig:coframing_identities}
\end{figure}

Fixing $v_0,v_1$ uniformly for all biframed knotoid diagrams prevents the case of biframed knotoid diagrams that are in-equivalent only because their end-points disagree and cannot be moved. An example of a biframed knotoid diagram is given in Figure \ref{fig:biframed_example}.

\begin{figure}[ht]
    \centering
    \includegraphics[width=.4\linewidth]{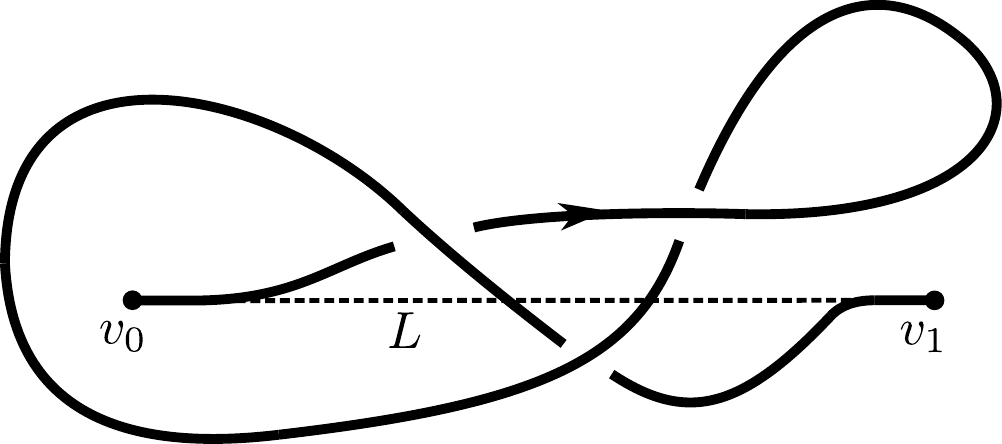}
    \caption{Example of a biframed knotoid diagram.}
    \label{fig:biframed_example}
\end{figure}

To give meaning to the above terminology, we define the framing and coframing of a biframed knotoid diagram. After Definitions \ref{def:writhe} and \ref{def:coframing}, there are no surprises here.

\begin{definition}\label{def:diagcoframing}
Let $K$ be a biframed knotoid diagram. We define its \textbf{framing} to be its writhe. To define its coframing, let $n_i$ be the winding number of $K$ around $v_i$ for $i\in \{0,1\}$. By assumption on the tangents of $K$, $n_i\in \mathbb{Z}$. In analogy with Equation \eqref{eq:coframing} we define the \textbf{coframing} of $K$ by
\[
    \text{cofr}(K) = n_0-n_1.
\]
\end{definition}

To see that the coframing is well-defined we must verify that it is invariant under the moves on biframed knotoid diagrams. This is the content of the following lemma.

\begin{lemma}
The coframing is a biframed knotoid diagram invariant.
\end{lemma}
\begin{proof}
The only moves for which invariance is not immediate are the non-local moves, i.e.~$R4$ and the coframing identities. The coframing is defined to be invariant under the coframing identities, namely these identities respectively change both $n_0$ and $n_1$ by $\pm1$, and so these changes cancel in the definition of $\text{cofr}(K)$. Next say that $R4$ is applied to an arc that traverses $\alpha_i$ radians around $v_i$. Then the new arc traverses $-(2\pi-\alpha_i)$ radians around $v_i$. Noting that
\[
    -(2\pi-\alpha_0) + (2\pi-\alpha_1) = \alpha_0-\alpha_1,
\]
we see that these differences again cancel in the definition of $\text{cofr}(K)$.
\end{proof}

\begin{lemma}\label{lm:biframedknotoidbijection}
There is a bijection
\begin{align}
    \{ \text{Biframed knotoids} \} &\leftrightarrow \{ \text{Knotoids} \}\times\mathbb{Z}^2 \label{eq:biframedknotoidbijection}\\
    K &\mapsto \left(K,\text{writhe}(K),\text{cofr}(K)\right) \nonumber.
\end{align}
\end{lemma}

\begin{proof}
Clearly $n_0$, $n_1$, and hence the coframing of a biframed knotoid, can be adjusted at will by adding loops around $v_0$ as needed. Combined with bijection \eqref{eq:framedknotoidbijection}, this gives surjectivity of assignment \eqref{eq:biframedknotoidbijection}. Injectivity follows from bijection \eqref{eq:framedknotoidbijection}, and reasoning analogous to the proof of bijection \eqref{eq:framedknotoidbijection} for the coframing: in short, an equivalence $K\simeq K'$ as knotoids can be extended to an equivalence of biframed knotoids, up to loops around $v_0,v_1$ which must cancel via Figures \ref{fig:coframing_identities} and \ref{fig:cancelling_coframing} if we assume $\text{cofr}(K)=\text{cofr}(K')$. Thus assignment \eqref{eq:biframedknotoidbijection} is a bijection.
\end{proof}

\begin{figure}[ht]
    \centering
    \includegraphics[width=.7\linewidth]{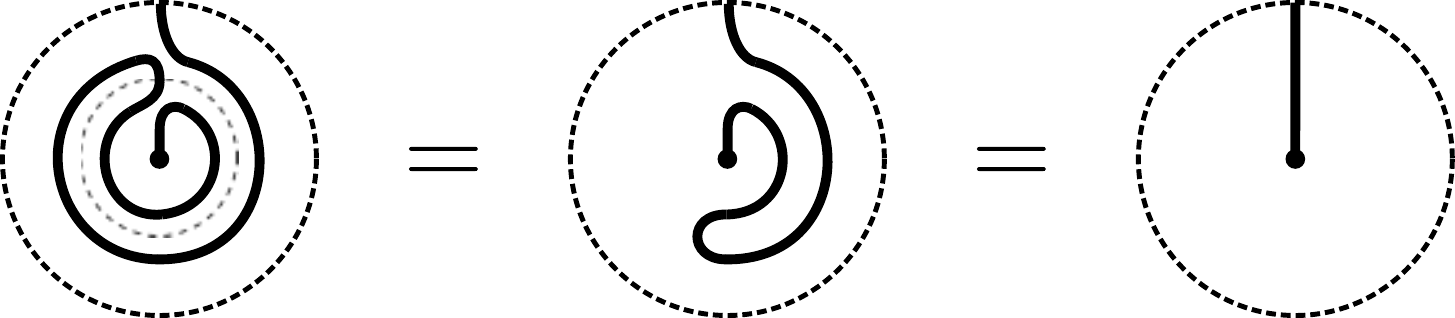}
    \caption{Cancellation of opposite coframing loops.}
    \label{fig:cancelling_coframing}
\end{figure}

Combined with bijection \eqref{eq:biframedthetabijection} and Theorem \ref{thm:knotoidbijection} this results in a geometric realization result for biframed knotoids analogous to bijection \eqref{eq:framedrealization}; namely a bijection
\begin{equation}\label{eq:biframedrealizaiton}
    \{\text{Biframed knotoids}\} \leftrightarrow \{\text{Biframed simple theta-curves}\},
\end{equation}
casting $\lvert\Theta\rvert$ as the geometric realization of biframed knotoids.

As an extension to Remark \ref{rk:bijection}, bijection \eqref{eq:biframedrealizaiton} also has a topological realization analogous to the `projection' $\theta\mapsto e_0$ from Figure \ref{fig:bijection}. The respective definitions of framing for diagrams and theta-curves translate to each other under this projection via the blackboard framing as before, and the respective definitions of coframing are set up so that they translate to each other immediately under this projection.

\subsubsection{Half-integer Coframing}

Our original motivation for introducing biframed theta-curves was finding a way of adding a `half-twist' to a framed knotoid in such a way that can be seen easily in the corresponding knotoid diagrams. Here we discuss how biframed knotoids provide this feature by extending the discussion to non-integer coframing.\\

Biframed knotoids with half-integer coframing are obtained simply by weakening the condition of simplicity of a biframed theta-curves by removing the requirement that the coframing be integer. This adds the possibility of theta-curves with $\theta(\overline{e}_0)$ attached to $\theta(\overline{e}_-\cup\overline{e}_+)$ at the \textit{back}. See Figure \ref{fig:halfinteger}. 

\begin{figure}[ht]
    \centering
    \includegraphics[width=.5\linewidth]{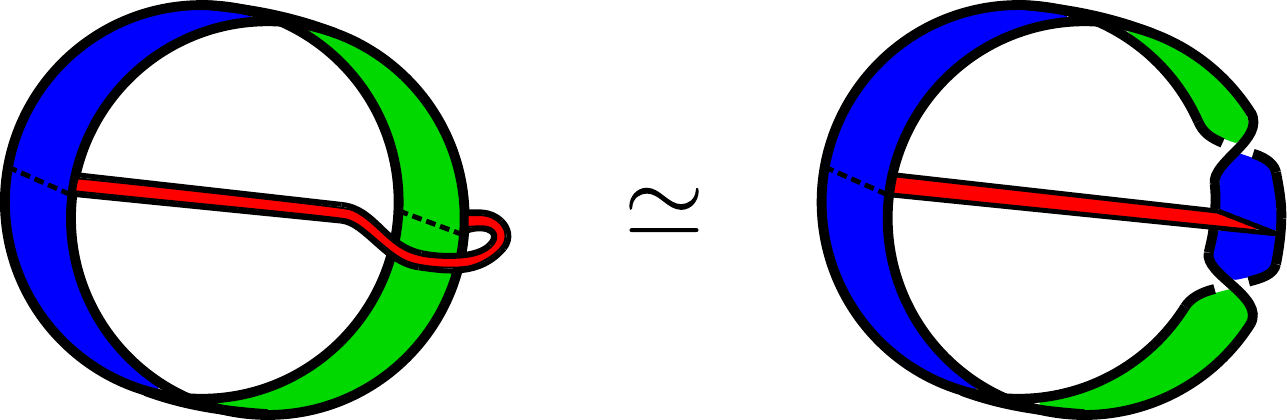}
    \caption{A standard biframed theta-curve with half-integer coframing.}
    \label{fig:halfinteger}
\end{figure}

Such an attachment at the back can occur at either end of $\theta(\overline{e}_0)$, and we can translate between these two cases via the isotopies shown in Figure \ref{fig:halfint_welldefinedness}. These same isotopies show that a biframed theta-curve with two half-integer coframing attachments is nothing but an integer-coframing theta-curve, as the terminology `half-integer' suggests.

\begin{figure}[ht]
    \centering
    \includegraphics[width=\linewidth]{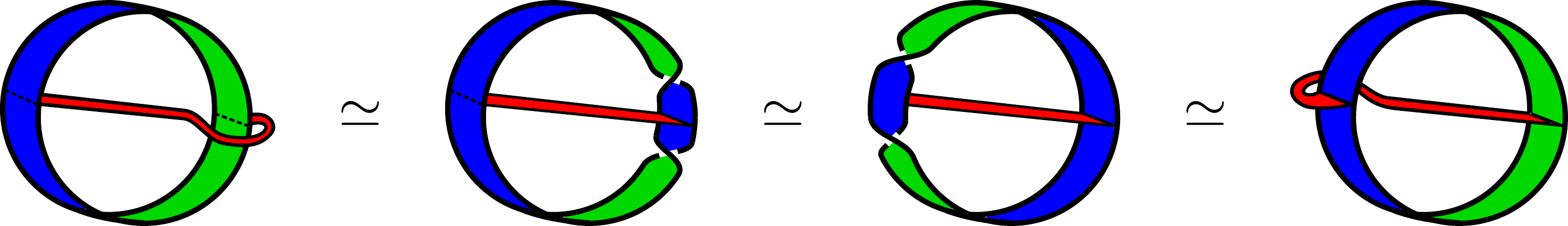}
    \caption{Translating a half-integer coframing attachment from $a_1$ to $a_0$. (A re-coloring occurs in the last equivalence.)}
    \label{fig:halfint_welldefinedness}
\end{figure}

The coframing of a half-integer coframing theta-curve is defined analogously as for the integer-coframing case, i.e.~by putting the theta-curve in standard form and subtracting the turning numbers of $\theta(\overline{e}_0)$ around its attachments to $\theta(\overline{e}_-\cup\overline{e}_+)$. 

Half-integer coframing manifests in biframed knotoid diagrams as the possibility of the tangent vectors at the end-points to be \textit{opposite} to $L$. See Figure \ref{fig:halfint_diagram}. (Again, the coframing in this case is defined as a difference of turning numbers, exactly as for integer-coframing knotoid diagrams.)

\begin{figure}[ht]
    \centering
    \includegraphics[width=.7\linewidth]{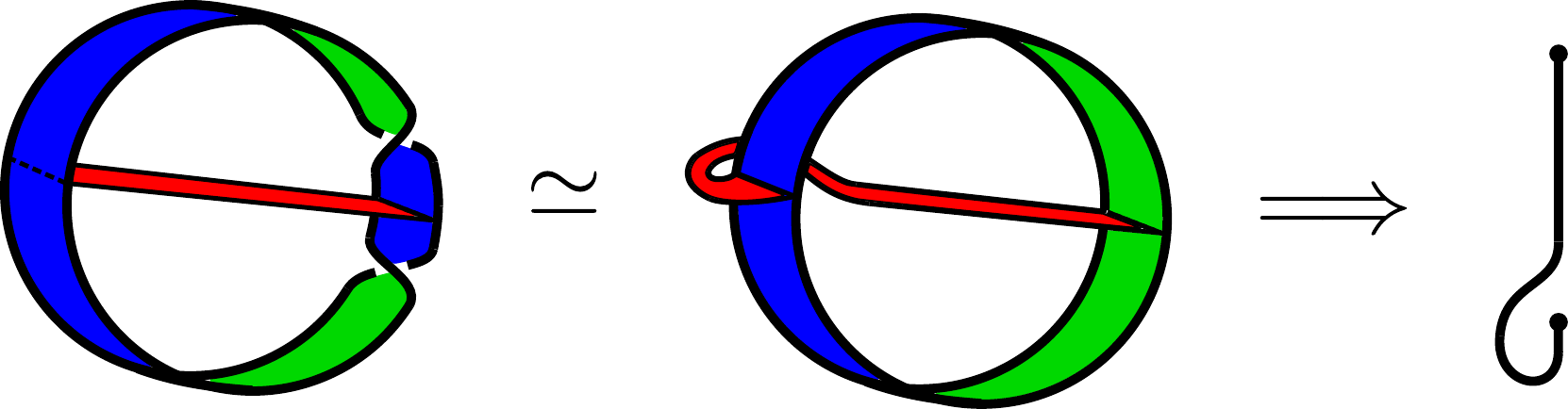}
    \caption{A $(0,-1/2)$-biframed trivial knotoid diagram.}
    \label{fig:halfint_diagram}
\end{figure}

To properly define half-integer coframing knotoids, then, we need to add coframing identities encoding the moves in Figure \ref{fig:halfint_welldefinedness} and similar such moves.

\begin{figure}[ht]
    \centering
    \includegraphics[width=\linewidth]{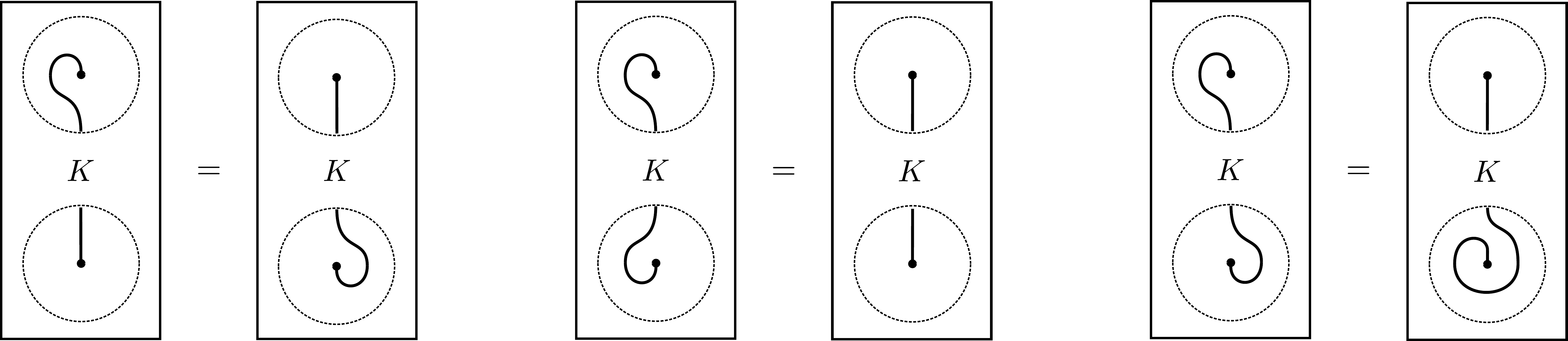}
    \caption{Additional coframing identities for half-integer coframing knotoids, for coframing $+\frac{1}{2}$, $0$, and $+1$ respectively. Here $K$ denotes the rest of a knotoid diagram as before.}
    \label{fig:halfint_coframing_ids}
\end{figure}

There are a total of 6 coframing identities that need to be added; the three identities involving a positive half-integer coframing attachment at the leg are depicted in Figure \ref{fig:halfint_coframing_ids}.

% The sense in which we have added the possibility of a half-twist, then, is that for a half-integer coframing theta-curve $\theta$,
% \[
%     \theta(\overline{e_+}\cup\overline{e_0}) \cong M \cong \theta(\overline{e_0}\cup\overline{e_-}),
% \]
% where $M$ is the M\"{o}bius band.

\section{Quantum Knotoid Invariants}\label{sec:invariants}

In this section we begin by recalling the construction of quantum knot invariants, specifically Reshetikhin-Turaev invariants, and then give several generalizations of such invariants to the setting of knotoids using framed and biframed knotoid diagrams. One such generalization relies on a slight modification to the setup of quantum invariants, moving to the setting of `braided groups'. Hence we first recall this setup very briefly.

\subsection{Reshetikhin-Turaev Invariants}

We follow standard notation from \cite{majidprimer}, where $\mathcal{C}$ is a category, $\otimes$ provides a monoidal structure on $\mathcal{C}$ with unit $k$, $(\Psi,\nu)$ further provide a ribbon braided structure on $\mathcal{C}$, and $(\cdot)^*$ is a dualization operation on $\mathcal{C}$ that turns it into a rigid category via $(\text{ev},\text{coev})$. Recall that quantum knot invariants are constructed by fixing a rigid ribbon category $(\mathcal{C},\otimes,k,\Psi,(\cdot)^*,\text{ev},\text{coev},\nu)$, picking an object $V$ of $\mathcal{C}$, and mapping a knot diagram to an element of $\Hom_{\mathcal{C}}(k,k)$ via the assignment specified on elementary tangles by Figure \ref{fig:braided_elements}. This assignment, call it $\varphi$, is extended to arbitrary tangle diagrams via the rules $\varphi(T_1\otimes T_2)=\varphi(T_1)\otimes \varphi(T_2)$ and $\varphi(T_1\circ T_2)=\varphi(T_1)\circ\varphi(T_2)$, where the tensor product and composition of tangles denote their horizontal and vertical juxtapositions respectively. For a more detailed discussion see \cite{jackson}.

\begin{figure}[ht]
    \centering
    \includegraphics[width=\linewidth]{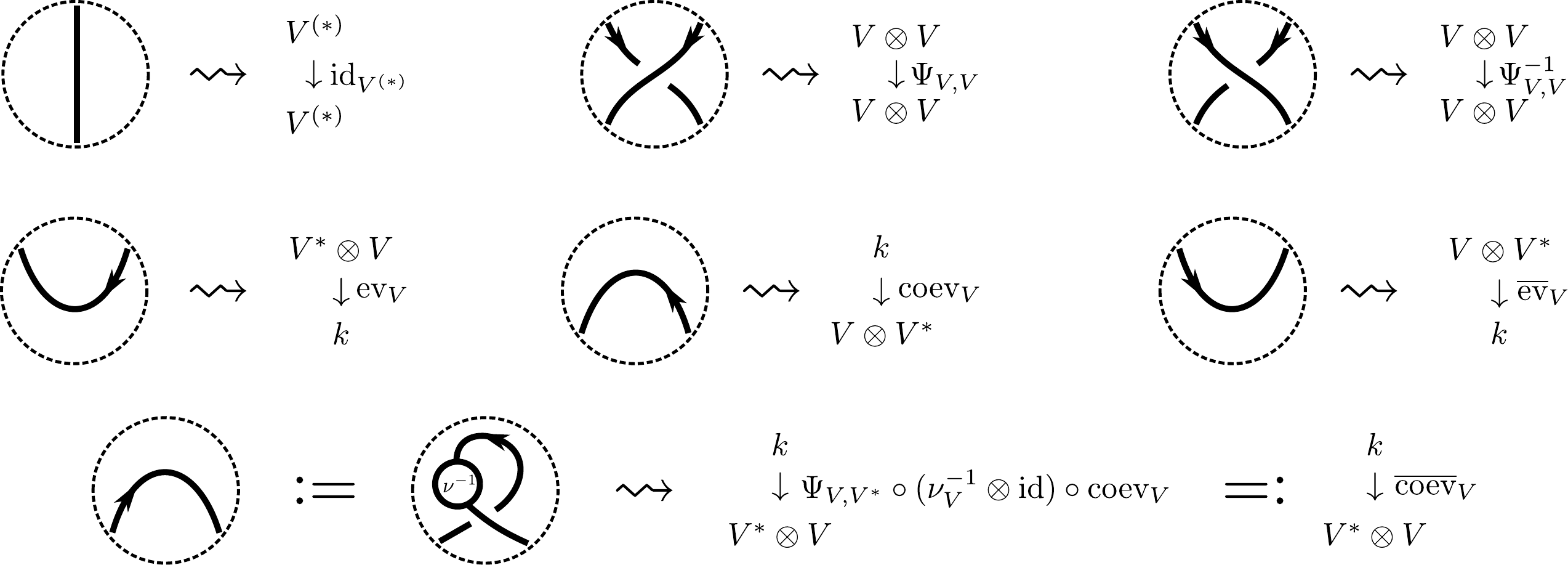}
    \caption{The assignment $\varphi$ on elementary tangles.}
    \label{fig:braided_elements}
\end{figure}

Here $\overline{\text{ev}}_V$ and $\overline{\text{coev}}_V$ are defined using the ribbon element $\nu_V^{-1}$, as is detailed for $\overline{\text{coev}}_V$. Our convention is that knot diagrams are read top to bottom, with downward strands represented by $V$ and upward strands by $V^*$. More generally for crossings that involve $V^*$, $\Psi$ is assigned to positively oriented crossings and $\Psi^{-1}$ to negative ones.

The assignment $\varphi$ is well-defined on framed tangles, and therefore provides an invariant of oriented framed knots for any choice $(\mathcal{C},V)$ when applied to a knot diagram viewed as a tangle with no in- or out-going strands. Such an invariant is computed by selecting a `\textit{Morse decomposition}' of the knot; dividing it into a composition of tangles, each of which is a finite tensor product of elementary tangles.

For our purposes here, we will wish to extend $\varphi$ to the setting of tangle diagrams decorated with splittings and mergings of strands. To this end, consider a Hopf algebra $(H,\cdot,\Delta,\eta,\epsilon,S)$, following standard notation from \cite{majidfoundations}. Then in analogy with $\varphi$ we can represent morphisms $H^{\otimes n}\to H^{\otimes m}$ as tangles decorated with splittings, mergings, and end-points as in Figure \ref{fig:hopf_elements}.

\begin{figure}[ht]
    \centering
    \includegraphics[width=.8\linewidth]{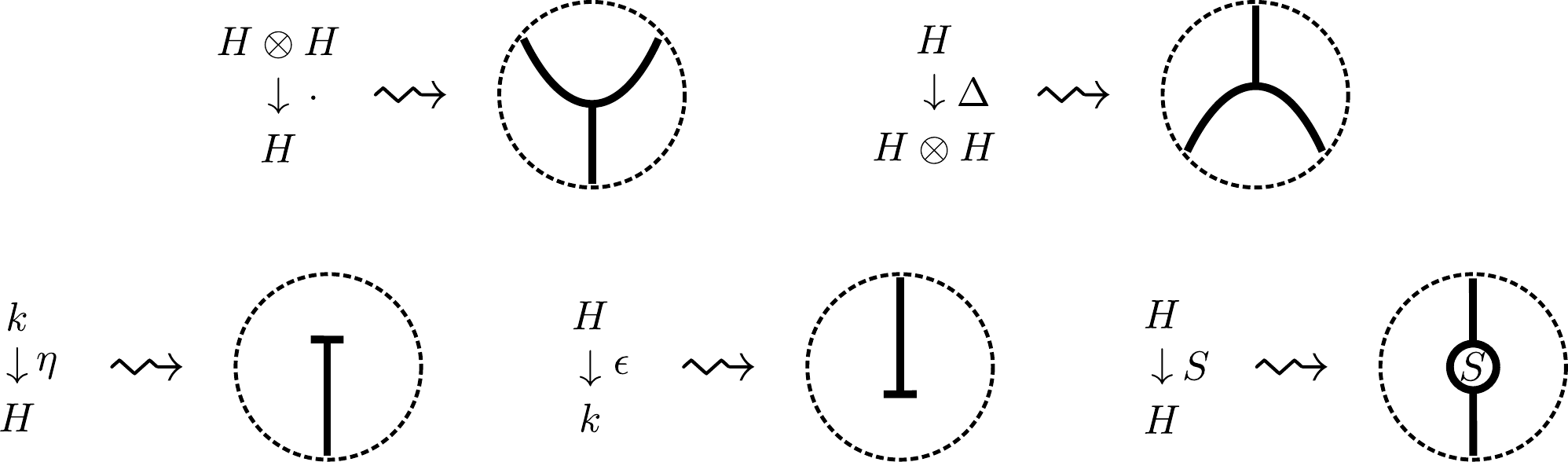}
    \caption{Diagrammatic depiction of structure morphisms of $H$.}
    \label{fig:hopf_elements}
\end{figure}

To illustrate this, the resulting diagrammatic identities representing the axioms of a Hopf algebra are depicted in Figure \ref{fig:hopf_axioms}.

\begin{figure}[ht]
    \centering
    \includegraphics[width=\linewidth]{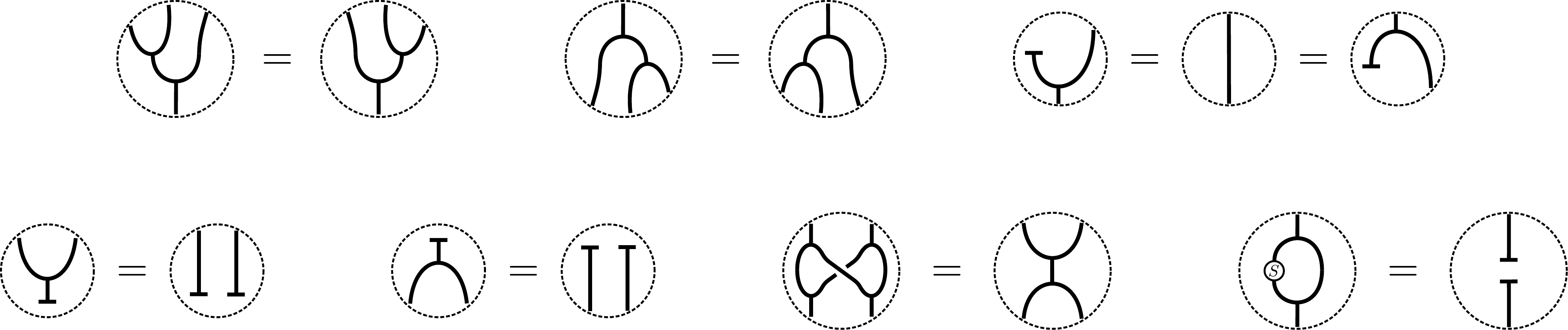}
    \caption{The axioms of a Hopf algebra in diagrammatic notation.}
    \label{fig:hopf_axioms}
\end{figure}

Notice that the second-to-last axiom in Figure \ref{fig:hopf_axioms} contains a crossing. For Hopf algebras this represents the vector space isomorphism $\tau:V\otimes W\cong W\otimes V$, i.e.~the trivial braiding on the category $\texttt{Vect}_k$. The diagrammatic notation developed so far now suggests we define the following:

\begin{definition} \cite{majid1993beyond,majidfoundations}
Let $\mathcal{C}$ be a ribbon category. A \textbf{braided group} (or `Hopf algebra object of a ribbon category') is an object $B$ of $\mathcal{C}$ endowed with morphisms $\cdot:B\otimes B\to B$, $\Delta: B\to B\otimes B$, $\eta:k\to B$, $\epsilon: B\to k$, $S:B\to B$. These morphisms must be such that the axioms depicted in Figure \ref{fig:hopf_axioms} are satisfied when considered as equations of morphisms via Figures \ref{fig:braided_elements} and \ref{fig:hopf_elements}. This means the crossing in Figure \ref{fig:hopf_axioms} is now taken to represent $\Psi$ instead of $\tau$, in accordance with Figure \ref{fig:braided_elements}.
\end{definition}

If $B$ is a braided group in $\mathcal{C}$, then the inverse of Figure \ref{fig:hopf_elements} provides an extension of $\varphi$ to tangle diagrams decorated with splittings, mergings, and end-points. This extension is well-defined on decorated framed tangles because $\varphi$ is, and because we can pull the decorations under and over crossings by naturality of $\Psi,\Psi^{-1}$.

Finally recall that to produce explicit Reshetikhin-Turaev invariants one usually picks a ribbon Hopf algebra $H$ and picks $\mathcal{C}={}_H\mathcal{M}_F$, the category of finite-dimensional $H$-modules. By the representation theory of Hopf algebras this is a rigid ribbon category \cite[~Ch.~9]{majidfoundations}. A closely related choice is $\mathcal{C}={}_H^H\mathcal{M}_F$, the category of finite-dimensional crossed $H$-modules, which is equivalent to ${}_{D(H)}\mathcal{M}_F$ for $H$ finite-dimensional where $D(H)$ is the Drinfeld double of $H$ \cite[~Ch.~10]{majidprimer}.

\subsection{Group-Based Invariants}

In this subsection we generalize quantum invariants resulting from seeing a finite-dimensional Hopf algebra $H$ as a crossed module in ${}_H^H\mathcal{M}_F={}_{D(H)}\mathcal{M}_F$. We restrict to $H=kG$ for $k$ a field and $G$ a finite group because it will give us an additional knotoid invariant for free, but the contents of this subsection generalize readily to arbitrary (finite-dimensional) $H$ as in \cite{vanderveen}.\\

For $H=kG$, we first describe an alternative way of computing the framed knot invariant associated to $H\in {}_H^H\mathcal{M}_F$. Call this invariant $V_G$. Note that $D(kG)=(kG)^*\otimes kG\cong k(G)\otimes kG$, where $k(G)$ is the group of functions $G\to k$, and multiplication in $D(kG)$ is given by
\begin{equation}\label{eq:DkGmultiplication}
    (\delta_a\otimes h)(\delta_b\otimes g)(x)= \delta_a\delta_{hbh^{-1}}\otimes hg
\end{equation}
where $g,h,a,b\in G$, and $\delta_a$ is the Kronecker delta function taking value 1 on $a$. Moreover $D(kG)$ has quasitriangular structure $\mathcal{R}$ given by
\[
    \mathcal{R} = \sum_{g\in G} (\delta_{g}\otimes e)\otimes (1\otimes g)
    \qquad\text{ and }\qquad
    \mathcal{R}^{-1} = \sum_{g\in G} (\delta_{g^{-1}}\otimes e)\otimes (1\otimes g).
\]
Here $e\in G$ is the group unit. For more details see \cite[~Ch.~7]{majidfoundations}. Note that each term of $\mathcal{R}\in D(H)\otimes D(H)$ consists of a `left part' $(\delta_g\otimes e)$ and a `right part' $(1\otimes g)$; similarly for $\mathcal{R}^{-1}$.\\

Now let $K$ be a framed knot. Present $K$ as an oriented long knot diagram. Then $V_G(K)$ is an element of $\Hom(kG,kG)\cong (kG)^*\otimes kG=D(kG)$. To compute this element we place at every crossing of $K$ a copy of $\mathcal{R}$ if the crossing is positive, or a copy of $\mathcal{R}^{-1}$ if it is negative. We then trace through the long knot diagram of $K$ following its orientation. Beginning with the unit $1\otimes e\in D(kG)$, every time we encounter a crossing we multiply (from the left) with the left part of $\mathcal{R}$ or $\mathcal{R}^{-1}$ if we are going under the crossing, or with the right part if we are going over. After the final crossing we sum over all indices of all crossings, and the result is equal to $V_G(K)$ \cite{vanderveen}. To obtain the knot invariant associated to this long knot invariant one takes the trace of $V_G(K)\in\Hom(kG,kG)$. The result is the cardinality of $\Hom(\pi_1(S^3-K),G)$ \cite[~Ch.~13]{majidprimer}, for reasons we shall soon discuss.\\

As we shall see in the next subsection, the general Reshetikhin-Turaev invariant construction is tricky to generalize to knotoids. However, the above algorithm for computing $V_G$ generalizes effortlessly to give a framed knotoid invariant:

\begin{definition}
Let $K$ be a diagram for a framed knotoid. Recall that all knotoids are implicitly oriented from leg to head. Define $V_G(K)\in D(H)$ by the same process as above, tracing through $K$ from leg to head.
\end{definition}

\begin{proposition}
The element $V_G(K)$ is a framed knotoid invariant. 
\end{proposition}

\begin{proof}
Invariance under ambient isotopy and the spherical move $R4$ is immediate since these moves do not alter the crossing data of $K$. Thus it suffices to check invariance of $V_G(K)$ under $R1',R2,R3$. This follows from $V_G$ being an invariant when defined on framed (long) knots. For completeness, we also prove by hand that $V_G$ is a framed knotoid invariant below:

Since the oriented versions of $R1'$ do not change the signs of crossings, invariance of $V_G(K)$ under $R1'$ follows from
\[
    (\delta_g\otimes e)(1\otimes g) = \delta_g\otimes g = (1\otimes g)(\delta_g\otimes e)
    \hspace{0.5em}\text{ and }\hspace{0.5em}
    (\delta_{g^{-1}}\otimes e)(1\otimes g) = \delta_{g^{-1}}\otimes g = (1\otimes g)(\delta_{g^{-1}}\otimes e).
\]
The move $R2$ has four oriented versions. The move $R3$ has eight. For brevity we consider only one of each; see Figure \ref{fig:orientedmoves}. The proofs for the other orientations are analogous.
\begin{figure}[h!]
    \centering
    \includegraphics[width=.8\linewidth]{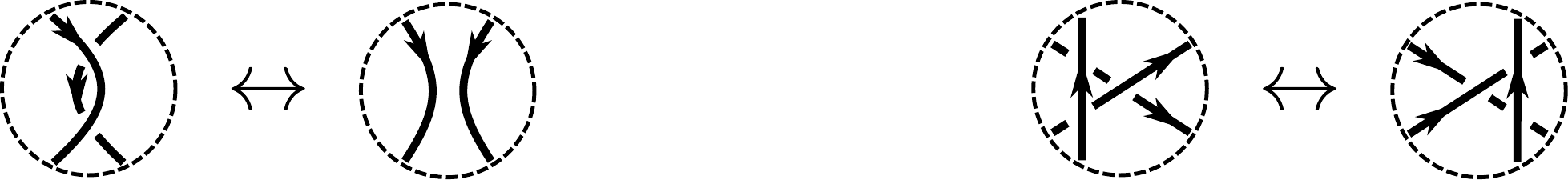}
    \caption{Examples of oriented $R2$ and $R3$ moves.}
    \label{fig:orientedmoves}
\end{figure}
On the left-hand side of the $R2$ move in Figure \ref{fig:orientedmoves} there are two crossings. For each orientation of $R2$, one is positive and the other negative. To find invariance we must show that the contributions from these crossings cancel. Indeed, the bottom arc gives a contribution
\[
    (\delta_{g^{-1}}\otimes e)(\delta_h\otimes e)
    =
    \begin{cases}
    \delta_h\otimes e &\quad \text{ if } g^{-1}=h,\\
    0 &\quad \text{ otherwise.}
    \end{cases}
\]
The top arc gives a contribution $(1\otimes g)(1\otimes h)$ which is thus $(1\otimes e)$ for nonzero contributions. Noting that $\sum_h(\delta_h\otimes e)=1\otimes e$ we see that the contribution of the bottom arc is also trivial, as required.

To verify invariance under the $R3$ move in Figure \ref{fig:orientedmoves}, we calculate the contributions of both sides and show they are equal. When running through $K$, suppose we pass through the vertical line first, then through the line going diagonally up, and finally through the strand going diagonally down (other orders are proven analogously). Then we compute $V_G(K)$ on the left-hand side to be
\begin{align*}
    V_G(K) &= \sum_{g,h,k,\dots}\dots(1\otimes g)(1\otimes h)\dots(\delta_{g^{-1}}\otimes e)(1\otimes k)\dots(\delta_{h^{-1}}\otimes e)(\delta_{k^{-1}}\otimes e)\\
    &= \sum_{g,k,\dots}\dots(1\otimes g)(1\otimes k)\dots(\delta_{g^{-1}}\otimes e)(1\otimes k)\dots(\delta_{k^{-1}}\otimes e).
\end{align*}
Here the $(\dots)$ are generic placeholder for contributions due to the rest of $K$. Now we drag the term $(\delta_{k^{-1}}\otimes e)$ towards the middle contribution. Every time an element of the form $(1\otimes x)$ is encountered, the argument of $\delta_{k^{-1}}$ is conjugated by $x^{-1}$. Thus we find
\begin{align*}
    V_G(K) &= \sum_{g,k,\dots}\dots(1\otimes g)(1\otimes k)\dots(\delta_{g^{-1}}\otimes e)(1\otimes k)(\delta_{wk^{-1}w^{-1}}\otimes e)\dots\\
    &= \sum_{g,k,\dots}\dots(1\otimes g)(1\otimes k)\dots(\delta_{g^{-1}}\otimes e)(\delta_{kwk^{-1}w^{-1}k^{-1}}\otimes e)(1\otimes k)\dots,
\end{align*}
for some word $w$ on $G$. This implies that $g^{-1}=kwk^{-1}w^{-1}k^{-1}$. Thus we conclude
\[
    V_G(K) = \sum_{k,\dots} \dots(1\otimes kwkw^{-1})\dots(\delta_{kwk^{-1}w^{-1}k^{-1}}\otimes k)\dots
\]
Following the same procedure on the right-hand side (with the labels $g,h,k$ attached to the same crossings as before) we find
\begin{align*}
    V_G(K) &= \sum_{g,k,\dots}\dots(1\otimes k)(1\otimes g)\dots(1\otimes k)(\delta_{g^{-1}}\otimes e)\dots(\delta_{k^{-1}}\otimes e)\\
    &= \sum_{g,k,\dots}\dots(1\otimes k)(1\otimes g)\dots(1\otimes k)(\delta_{g^{-1}}\otimes e)(\delta_{wk^{-1}w^{-1}}\otimes e)\dots\\
    &= \sum_{k,\dots} \dots(1\otimes kwkw^{-1})\dots (1\otimes k)(\delta_{wk^{-1}w^{-1}}\otimes e)\dots\\
    &= \sum_{k,\dots} \dots(1\otimes kwkw^{-1})\dots(\delta_{kwk^{-1}w^{-1}k^{-1}}\otimes k)\dots.
\end{align*}
(Note that $w$ here is the same word as in the left-hand side computation!) We thus conclude that $V_G$ is indeed invariant under $R3$, completing the proof.
\end{proof}

\begin{example}\label{ex:knotoidgroup}
% Before proving that $V_G$ is indeed an invariant, we illustrate its construction with an example. 
Let $K$ be the framed knotoid depicted in Figure \ref{fig:VGexample}.
\begin{figure}[ht]
    \centering
    \includegraphics[width=.45\linewidth]{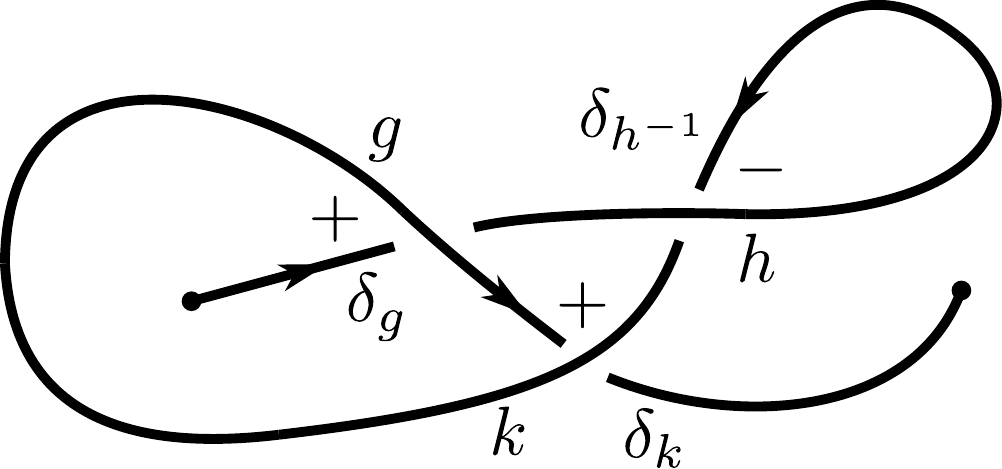}
    \caption{Construction of $V_G$ for an example framed knotoid.}
    \label{fig:VGexample}
\end{figure}
The signs of the crossings of $K$ are indicated in Figure \ref{fig:VGexample}, as are the nontrivial parts of $\mathcal{R},\mathcal{R}^{-1}$ that occur in the construction of $V_G$. Using the definition of $V_G$ and Equation \eqref{eq:DkGmultiplication} we compute:
\begin{align*}
    V_G(K) &= \sum_{g,h,k\in G} (\delta_g \otimes e)(1\otimes h)(\delta_{h^{-1}}\otimes e)(1\otimes k)(1\otimes g)(\delta_k\otimes e)\\
    &= \sum_{g,h,k\in G} \delta_g\delta_{h^{-1}} \delta_{hkgkg^{-1}k^{-1}h^{-1}} \otimes hkg\\
    &= \sum_{g,k\in G} \delta_g\delta_{g^{-1}kgkg^{-1}k^{-1}g} \otimes g^{-1}kg\\
    &= \sum_{g,k\in G} \delta_{gkg}\delta_{kgk} \otimes g^{-1}kg.
\end{align*}
Here the last two equalities use that $\delta_a\delta_b\neq 0\iff a=b$. This also means that a choice of $(g,k)$ only gives a nonzero contribution to $V_G(K)$ if $gkg=kgk$. In other words, we only obtain a contribution to $V_G(K)$ if $B_3\to G:(x,y)\mapsto(g,k)$ constitutes a group homomorphism, where
\[
    B_3 = \langle x,y\,\vert\, xyx=yxy \rangle
\]
is the braid group on $3$ strands.
% Note: this means precisely that $B_3$ is the knotoid group of $K$!
\end{example}

% As a more substantial example, let $K_\pm$ be the knotoids in Figure \ref{fig:orientation}, which are equal up to orientation.

% \begin{figure}[h!]
%     \centering
%     \includegraphics[width=8cm]{orientation.pdf}
%     \caption{The knotoids $K_+$ (left) and $K_-$ (right).}
%     \label{fig:orientation}
% \end{figure}

% \begin{proposition}
% The knotoids $K_+$ and $K_-$ are in-equivalent.
% \end{proposition}

% This result is known, but nontrivial \cite{goundaroulis2019systematic}. Here we show it using $V_G$.

% \begin{proof}
% First note that $K_+$ and $K_-$ have the same blackboard framing, and so it suffices to show that $V_G(K_+)\neq V_G(K_-)$. We compute
% \[
%     V_G(K_+) = \sum_{g,h,k,i,j} \delta_g\delta_{hkhk^{-1}h^{-1}}\delta_{hkij^{-1}i^{-1}k^{-1}h^{-1}}\delta_{hkigji^{-1}j^{-1}g^{-1}i^{-1}k^{-1}h^{-1}}\delta_{hkigjk^{-1}j^{-1}g^{-1}i^{-1}k^{-1}h^{-1}}\otimes hkigj
% \]
% and
% \[
%     V_G(K_-) = \sum_{g,h,k,i,j} \delta_{g^{-1}}\delta_{h^{-1}}\delta_{kik^{-1}i^{-1}k^{-1}}\delta_{kihjh^{-1}i^{-1}k^{-1}}\delta_{kihgjij^{-1}g^{-1}h^{-1}i^{-1}k^{-1}}\otimes kihgj.
% \]
% Taking $G=S_3$ one finds that $V_{S_3}(K_+)$ only has the six trivial contributions $g=h=k=i=j$, while $V_{S_3}(K_-)$ has a total of 12 contributions. For example, one of the nontrivial contributions to $V_{S_3}(K_-)$ is
% \[
%     g=(1,2)\quad
%     h=(1,2)\quad
%     k=(1,3)\quad
%     i=(1,2)\quad
%     j=(2,3).
% \]
% Thus $V_G$ distinguishes $K_+$ and $K_-$ as claimed.
% \end{proof}

At the beginning of this subsection we stated that taking $H=kG$ gives us an extra knotoid invariant for free. Indeed, this follows from the relation between $V_G$ and the Wirtinger presentation of the knot group $\pi_1(S^3-K)$. See \cite[~Ch.~3]{rolfsen} for more on the Wirtinger presentation. Namely, if $K$ is a knot presented as a long knot diagram then after some simplification $V_G(K)$ is of the form
\[
    V_G(K) = \sum_{g,h,\dots} \delta_{\dots} \delta_{\dots} \dots \delta_{\dots} \otimes \dots.
\]
Here the subscripts of the $\delta$'s are words in $(g,h,\dots)$. We only get a non-zero contribution to $V_G$ for $g,h,\dots$ such that all these words coincide. The multiplication in $D(H)$ is set up precisely so that the relations corresponding to coincidence of these words are exactly the relations of the Wirtinger presentation. Indeed: if we encounter a term in $V_G$ such as
\[
    (\delta_k\otimes e)(1\otimes g)(\delta_h\otimes e),
\]
then to arrive at the reduced form for $V_G$ we must pull the $\delta_h$ term to the left. The multiplication in $D(H)$ prescribes that this term equals
\[
    (\delta_k\otimes e)(\delta_{ghg^{-1}}\otimes g) = \delta_k\delta_{ghg^{-1}} \otimes g.
\]
This conjugation by $g$ reflects exactly the action of going under the over-lying $g$-strand and back again in order to relate $h$ to $k$, which is precisely the setup of the Wirtinger presentation.

In conclusion, the invariant $V_G$ gives us an alternative way of computing the Wirtinger presentation. This implies that nonzero contributions to $V_G(K)$ correspond to homomorphisms $\pi_1(S^3-K)\to G$. Both of these results are well-known \cite{bar2019polynomial}. The novelty is that we have extended $V_G$ to framed knotoids, and so this interpretation of $V_G$ allows us to associate a group to knotoids.

\begin{definition}\label{def:knotoidgroup}
Let $K$ be a knotoid diagram with $n$ crossings. Suppose $V_G(K)$ is of the form
\[
    V_G(K) = \sum_{g_1,g_2,\dots,g_n} \delta_{w_1}\delta_{w_2}\dots\delta_{w_n} \otimes w_{n+1}
\]
where each $w_i$ is a word in $g_1,\dots g_n$. We define the \textbf{knotoid group} $W(K)$ of $K$ by the presentation:
\[
    W(K) = \langle g_1,g_2,\dots g_n \vert w_1=w_2=\dots=w_n \rangle.
\]
\end{definition}

\begin{remark}
One might expect the knotoid group $W(K)$ to be isomorphic to the fundamental group of $S^3$ with an arc that projects to give $K$ removed, in which case it is isomorphic to $S^3$ with an arc removed and hence trivial. Instead $W(K)$ is a group defined by the combinatorial information of $K$ and hence not necessarily related to any fundamental group. We shall see below that it is nontrivial in general. A knotoid group was also defined by Turaev directly via the Wirtinger presentation \cite{turaev2012knotoids}. Turaev's knotoid group, which I shall denote here by $W'(K)$, has one generator more than $W(K)$, and in \cite{turaev2012knotoids} it is shown that $W'(K)\cong\pi_1(S^3-K_-)$ where $K_-$ is the knot formed from a knotoid $K$ by connecting the endpoints via a strand that runs under all of $K$. We shall see in Example \ref{ex:W(K)} below that this last fact does not hold for $W(K)$, and hence $W(K)\ncong W'(K)$ in general.
\end{remark}

Since $V_G$ is a framed knotoid invariant, invariance of $W(K)$ under $R1',R2,R3$ is immediate. Note however that Definition \ref{def:knotoidgroup} makes no reference to framing. This is justified by the following lemma:

\begin{lemma}
The knotoid group $W(K)$ is $R1$-invariant.
\end{lemma}
\begin{proof}
Let \includegraphics[width=.7cm]{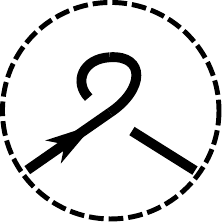} and \includegraphics[width=.7cm]{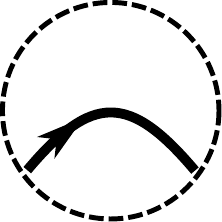} denote two knotoid diagrams for $K$ that are equal outside of the depicted discs. Let us write
\begin{align*}
    V_G\big( \raisebox{-.24cm}{\includegraphics[width=.7cm]{WKleft.pdf}} \big) 
    &= \sum_{g,g_1,g_2,\dots} \dots (1\otimes g)(\delta_{g^{-1}}\otimes e) \dots\\
    &= \sum_{\dots,g,g_1,\dots} \left( \delta_{\dots} \dots \delta_{\dots} \delta_{\Omega g^{-1}\Omega^{-1}} \delta_{\Omega g w_1 g_1 w_1^{-1} g^{-1} \Omega^{-1}} \delta_{\Omega g w_2 g_2 w_2^{-1} g^{-1} \Omega^{-1}}\dots \right) \otimes w'.
\end{align*}
Here $w',w_1,w_2,\dots$ are words in $g_1,g_2,\dots$, and $\Omega$ is the word by which the term $\delta_{g^{-1}}$ gets conjugated in the reduction of $V_G$. We want to show that the generator $g$ can be removed from the defining presentation of $W(K)$. Omitting the generators from the presentations, we compute
\begin{align*}
    W\big( \raisebox{-.24cm}{\includegraphics[width=.7cm]{WKleft.pdf}} \big)
    &= \langle \dots = \Omega g^{-1}\Omega^{-1} = \Omega g w_1 g_1 w_1^{-1} g^{-1} \Omega^{-1} = \Omega g w_2 g_2 w_2^{-1} g^{-1} \Omega^{-1} = \dots \rangle \\
    &= \langle \dots = \Omega g^{-1}\Omega^{-1}, \\
    & \qquad\qquad g^{-1} = g w_1 g_1 w_1^{-1} g^{-1} = g w_2 g_2 w_2^{-1} g^{-1} = \dots \rangle \\
    &= \langle \dots = \Omega g^{-1}\Omega^{-1}, \\
    & \qquad\qquad g^{-1} = w_1 g_1 w_1^{-1} = w_2 g_2 w_2^{-1} = \dots \rangle \\
    &= \langle \dots = \Omega w_1 g_1 w_1^{-1}\Omega^{-1}, \\
    & \qquad\qquad w_1 g_1 w_1^{-1} = w_2 g_2 w_2^{-1} = \dots \rangle \\
    &= \langle \dots = \Omega w_1g_1w_1^{-1}\Omega^{-1} = \Omega w_2g_2w_2^{-1}\Omega^{-1} = \dots \rangle
    = W\big( \raisebox{-.24cm}{\includegraphics[width=.7cm]{WKright.pdf}} \big),
\end{align*}
as required. The proof for the $R1$-relation of the other sign is analogous.
\end{proof}

\begin{example}\label{ex:W(K)}
For the knotoid diagram $K$ from Example \ref{ex:knotoidgroup}, we saw $W(K)=B_3$. Note that $K_-$ is equivalent to the unknot, and hence $W'(K)\cong \pi_1(S^3-K_-)\cong \mathbb{Z}$ where $W'(K)$ is the knotoid group defined in \cite{turaev2012knotoids}, so that $W'(K)\ncong W(K)$ in this case.
\end{example}

\begin{example}
As a more substantial example of computing $V_G$, we shall show that $V_G(K)$ can contain more information than the cardinality of $\text{Hom}(W(K),G)$. Indeed: let $K_1$ and $K_2$ be the oriented knotoids depicted in figure \ref{fig:5_5and24}. 

\begin{figure}[h!]
    \centering
    \includegraphics[width=.5\linewidth]{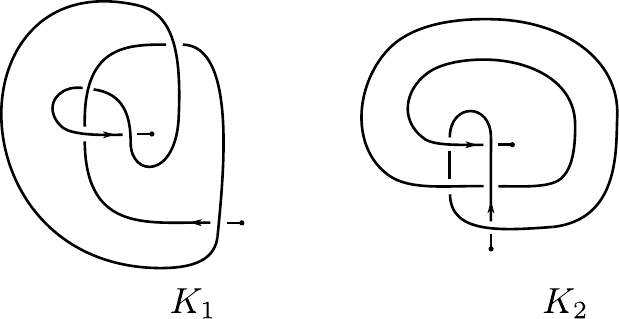}
    \caption{The knotoids $K_1$ and $K_2$, corresponding respectively to orientations of the knotoids $5_5$ and $5_{24}$ in the knotoid table from \cite{goundaroulis2019systematic}.}
    \label{fig:5_5and24}
\end{figure}

We compute that
\begin{align*}
    V_G(K_1) 
    &= \sum_{g,h,i,j,k\in G} (\delta_{g^{-1}}\otimes e) (\delta_{h}\otimes e) (1\otimes i) (\delta_{j}\otimes e) (1\otimes g)\\ 
    &\hspace{.2\linewidth}\cdot (1\otimes j) (1\otimes k) (\delta_{i}\otimes e) (1\otimes h) (\delta_{k^{-1}}\otimes e)\\
    &= \sum_{g,h,i,j,k\in G} \delta_{g^{-1}} \delta_{h} \delta_{iji^{-1}} \delta_{igjki(igjk)^{-1}} \delta_{igjkhk^{-1}(igjkh)^{-1}} \otimes igjkh.
\end{align*}
and
\begin{align*}
    V_G(K_2) 
    &= \sum_{g,h,i,j,k\in G} (\delta_{g}\otimes e) (1\otimes h) (1\otimes i) (\delta_{j^{-1}}\otimes e) (\delta_{k^{-1}}\otimes e)\\ 
    &\hspace{.2\linewidth}\cdot  (1\otimes g) (1\otimes k) (\delta_{h^{-1}}\otimes e) (1\otimes j) (\delta_{i^{-1}}\otimes e)\\
    &= \sum_{g,h,i,j,k\in G} \delta_{g} \delta_{hij^{-1}(hi)^{-1}} \delta_{hik^{-1}(hi)^{-1}} \delta_{higkh^{-1}(higk)^{-1}} \delta_{higkji^{-1}(higkj)^{-1}} \otimes higkj.
\end{align*}

Taking $G=S_3$, we find that both $V_{S_3}(K_1)$ and $V_{S_3}(K_2)$ have 12 nonzero contributions. More specifically, we find
\begin{align*}
    &V_{S_3}(K_1) = \delta_{e}\otimes e + \delta_{(1,3,2)}\otimes (1,3,2) + \delta_{(1,2,3)}\otimes (1,2,3) + \dots,\\
    &V_{S_3}(K_2) = \delta_{e}\otimes e + \delta_{(1,2,3)}\otimes e + \delta_{(1,3,2)}\otimes e + \dots,
\end{align*}
where the omitted terms in both expressions each consist of ten terms only involving the 2-cycles $(1,2),(1,3),(2,3)\in S_3$. From these computations we note that
\[
    \#\text{Hom}(W(K_1),S_3) = 12 = \#\text{Hom}(W(K_2),S_3),
\]
meaning that $\#\text{Hom}(W(\cdot),S_3)$ cannot distinguish $K_1$ and $K_2$. However, the expressions of $V_{S_3}(K_1)$ and $V_{S_3}(K_2)$ given above clearly show that $V_{S_3}(K_1)\neq V_{S_3}(K_2)$, which implies that $V_{S_3}(\cdot)$ can distinguish $K_1$ and $K_2$. We therefore conclude that $V_G(\cdot)$ can be a strictly stronger invariant than $\#\text{Hom}(W(\cdot),S_3)$, as claimed.

\end{example}

\subsection{Reshetikhin-Turaev Invariants of Knotoids}

There are two main obstructions to extending general Reshetikhin-Turaev invariants to framed knotoids. The first point is of course that of determining what to do with the end-points. The second obstacle is that for a (framed) knotoid diagram, one is able to swivel the tangent vector at the end-points to go, say, either up or down. Such swiveling affects the number of \raisebox{-.24cm}{\includegraphics[width=.7cm]{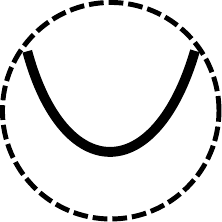}}'s and \raisebox{-.24cm}{\includegraphics[width=.7cm]{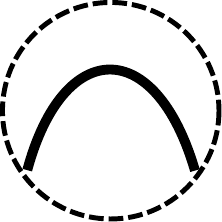}}'s in a Morse decomposition of the diagram. Unfortunately there is no reason that a Reshetikhin-Turaev-type assignment of morphisms to knotoid diagrams should be invariant under this action.

Fortunately, in light of Section \ref{sec:topology} we are equipped to solve this second problem. Namely, restricting our scope to producing invariants of \textbf{biframed} knotoids removes the obstruction entirely. A posteriori this is one of the main reasons for being interested in biframed knotoids in the first place!

Naturally, the first obstruction admits several solutions. Intuitively one would think to extend $\varphi$ to knotoid diagrams by assigning to \raisebox{-.24cm}{\includegraphics[width=.7cm]{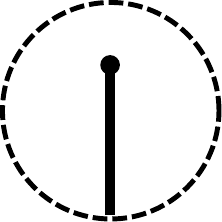}} and \raisebox{-.24cm}{\includegraphics[width=.7cm]{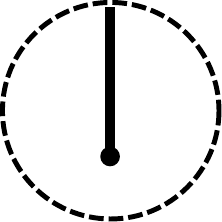}} morphisms $k\to V$ and $V\to k$ respectively. However, by naturality of $\Psi$ and $\Psi^{-1}$, $\varphi$ would then be invariant under the forbidden moves in Figure \ref{fig:forbiddenmoves} causing the associated invariant to be trivial.

Here we consider two options of how to handle \raisebox{-.24cm}{\includegraphics[width=.7cm]{start.pdf}} and \raisebox{-.24cm}{\includegraphics[width=.7cm]{finish.pdf}}.

\subsubsection{Option 1}

The next best thing to representing \raisebox{-.24cm}{\includegraphics[width=.7cm]{start.pdf}} and \raisebox{-.24cm}{\includegraphics[width=.7cm]{finish.pdf}} by morphisms $k\to V$ or $V\to k$, is to represent them by linear maps that are \textit{not} morphisms.

\begin{definition}\label{def:planarquantuminv}
Suppose $\mathcal{C}={}_H\mathcal{M}_F$ for some ribbon Hopf algebra\footnote{By Tannaka duality this is equivalent to $\mathcal{C}$ having a fiber functor \cite[~Ch.~9]{majidfoundations}.} $H$. Let $V$ be an $n$-dimensional object of $\mathcal{C}$, and $K$ a biframed knotoid diagram. For simplicity we assume $K$ has integer coframing. View $K$ as a tangle decorated with two endpoints representing maps $\eta:k\to V$ and $\epsilon: V\to k$. Extend $\varphi$ to an assignment on $K$, by choosing a basis $\{e_1,\dots,e_n\}$ of $V$ and defining
\[
    \eta : 1\mapsto \sum_{i=1}^n c_i e_i
    \qquad\text{ and }\qquad
    \epsilon : \sum_{i=1}^n \lambda_i e_i \mapsto \sum_{i=1}^n \lambda_i d_i,
\]
where $\{c_i,d_i\}_{i=1}^n$ consists of elements of $k$. The maps $\eta$ and $\epsilon$ are chosen to be linear maps, but \textit{not} $H$-module isomorphisms. The Reshetikhin-Turaev invariant associated to $K$ is then $\varphi(K)$.
\end{definition}

This solution is closely related to the approach taken in \cite{gugumcu2021quantum}. It has the advantage of being a simple and natural extension of $\varphi$. A distinct disadvantage is that this construction has no reason a priori to be invariant under the coframing identities, nor under the spherical move of Remark \ref{rk:spherical}. Thus in general this approach yields invariants of planar `triframed' knotoid diagrams, meaning that both endpoints have separate coframings which are not related by a coframing identity.

If we allow the end-points of these triframed knotoids to move around on $\mathbb{R}^2$, keeping the tangent vectors at the leg and head fixed, then we end up with the `Morse knotoids' from \cite{gugumcu2021quantum}. It turns out that the construction of quantum invariant given here also yields invariants of Morse knotoids, and that these Morse knotoids are essentially biframed planar knotoids in the sense that they can be described by knotoid type, framing integer, and `rotation number'. The rotation number is equivalent to the coframing as defined for diagrams in Definition \ref{def:diagcoframing}. See \cite{gugumcu2021quantum} for more details. To keep the discussion below self-contained, we will work with triframed knotoid diagrams.
% \footnote{Recalling the geometric interpretation of the coframing identities, this is exactly what one would expect from coframed \textit{planar} knotoids given the geometric realization of planar knotoids; see \cite{gugumcu2017new}.}

% In our case, we have $2n$ degrees of freedom in choosing $\eta,\epsilon$ and this can be used to produce extensions of $\varphi$ that \textit{are} invariant under the coframing identities.

\begin{example}
We consider the example of $H=U_q(\mathfrak{sl}_2)$, letting $V$ be the standard 2-dimensional representation. In this case the structures of $\Psi,\text{ev},\text{coev}$, etc. are well-known (up to a choice of normalization); see e.g.~\cite[~Ch.~2.6]{chmutov2012introduction} or \cite[~Ch.~9.4]{jackson}.
Let $\{e_1,e_2\}$ be the standard basis on $V$, $\{e^1,e^2\}$ the dual basis on $V^*$, and $\{e_1\otimes e_1,e_1\otimes e_2,e_2\otimes e_1,e_2\otimes e_2\}$ the standard ordered basis of $V\otimes V$ (and similarly for $V^*\otimes V$, $V\otimes V^*$). Then the values of $\varphi$ on knot diagram elements with respect to these bases are given in Figure \ref{fig:sl_2_elements}.

\begin{figure}[ht]
    \centering
    \includegraphics[width=.7\linewidth]{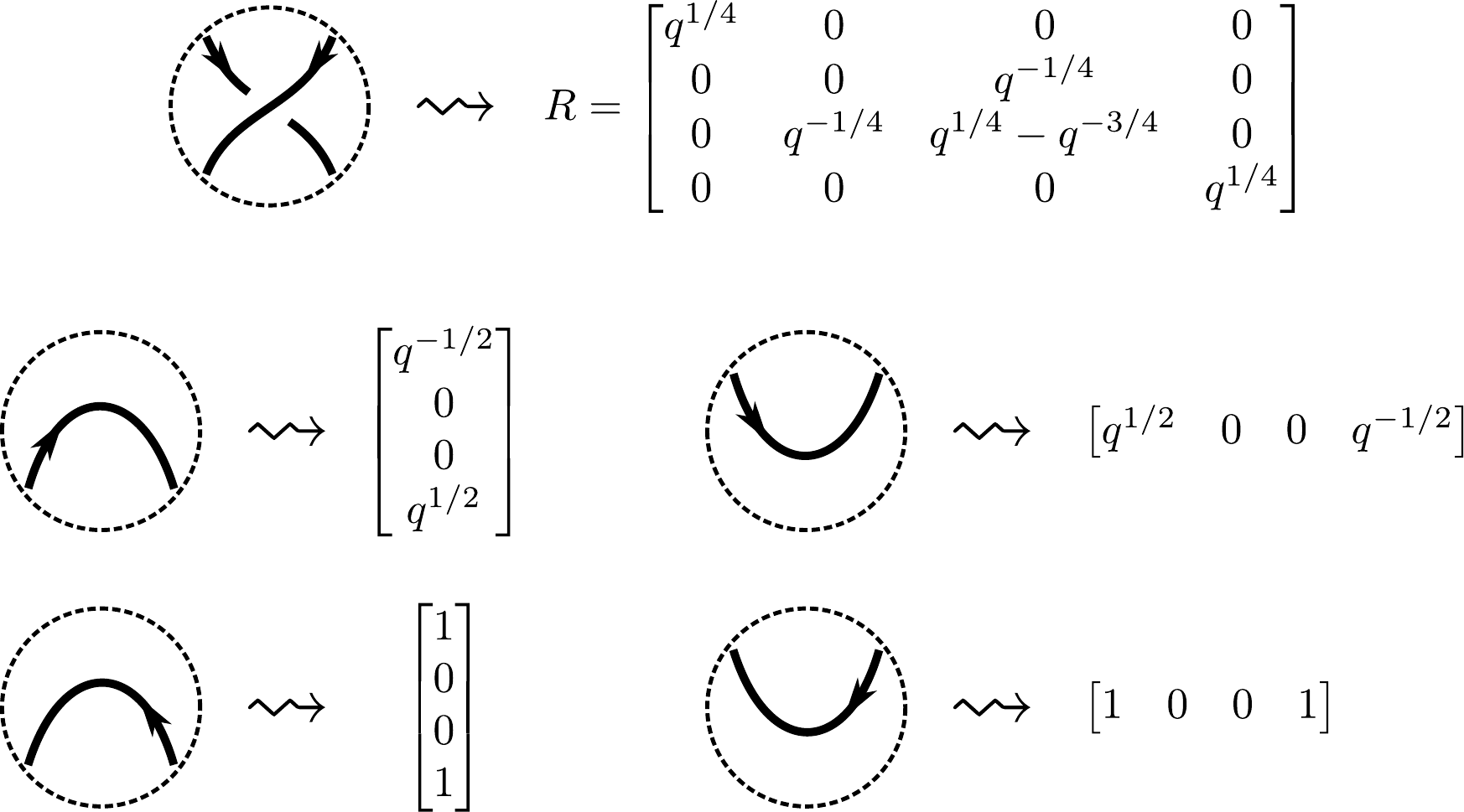}
    \caption{Values of $\varphi$ on knot diagram elements.}
    \label{fig:sl_2_elements}
\end{figure}

From now on, assume for simplicity that we are working with a (triframed) knotoid diagram $K$ with integer coframing, drawn so that the tangent vectors at the end-points of $K$ are \textit{vertical}.

We extend $\varphi$ to triframed knotoid diagrams by associating undetermined linear maps to the endpoints; see Figure \ref{fig:quantum_endpoints}.

\begin{figure}[ht]
    \centering
    \includegraphics[width=.6\linewidth]{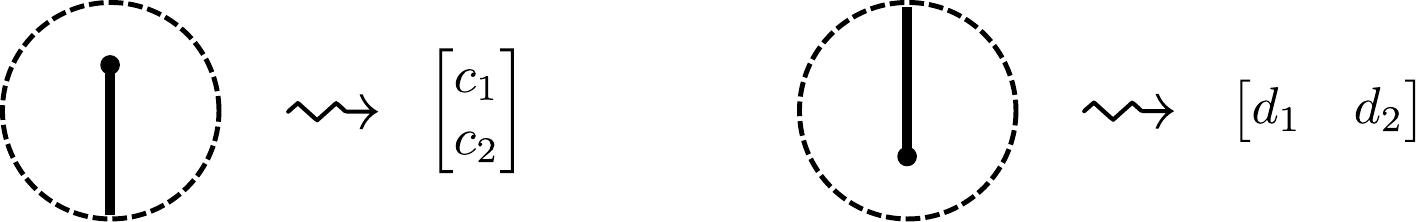}
    \caption{Values of $\varphi$ on remaining knotoid diagram elements.}
    \label{fig:quantum_endpoints}
\end{figure}

We can now carry out computations analogously as one does for knot diagrams, e.g.~computing that
\[
    q^{1/4} R - q^{-1/4} R^{-1} = (q^{1/2}-q^{-1/2})I_4,
\]
where $I_4$ is the identity matrix of dimension $4$. We conclude the following:

\begin{proposition}
Write $Q_{\mathfrak{sl}_2}$ to denote the assignment $\varphi$ for our choice of $(\mathcal{C},V)$. Then this assignment constitutes the existence of a class of triframed knotoid invariants, satisfying the properties of Figures \ref{fig:quantum_endpoints} and \ref{fig:sl_2_triframed_invariant}.

\begin{figure}[ht]
    \centering
    \includegraphics[width=.9\linewidth]{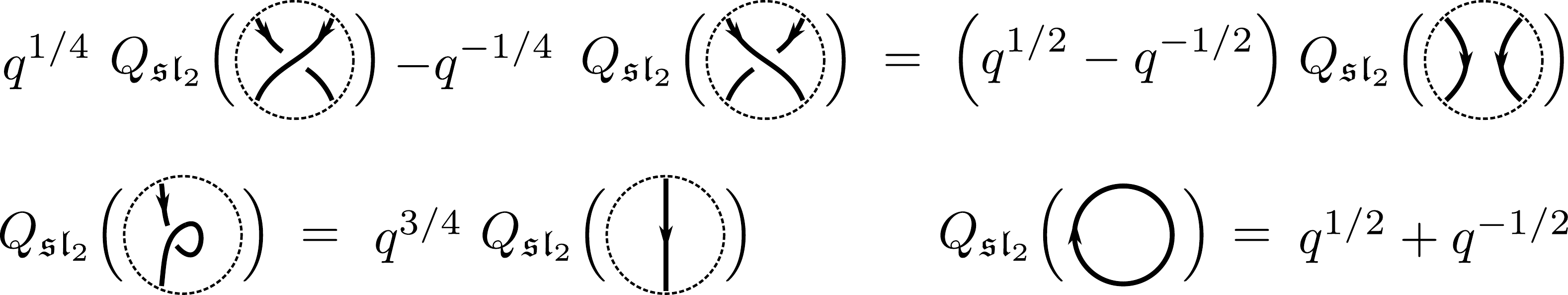}
    \caption{Skein relation and other properties of $Q_{\mathfrak{sl}_2}$.}
    \label{fig:sl_2_triframed_invariant}
\end{figure}

This class is indexed by the four free parameters $c_1,c_2,d_1,d_2$.
\end{proposition}

Our next question is whether we can pick $c_1,c_2,d_1,d_2$ so that $Q_{\mathfrak{sl}_2}$ is invariant under the coframing identities. To this end we calculate the value of $Q_{\mathfrak{sl}_2}$ on coframing loops around either end-points. See Figure \ref{fig:sl_2_coframing}.

\begin{figure}[ht]
    \centering
    \includegraphics[width=.85\linewidth]{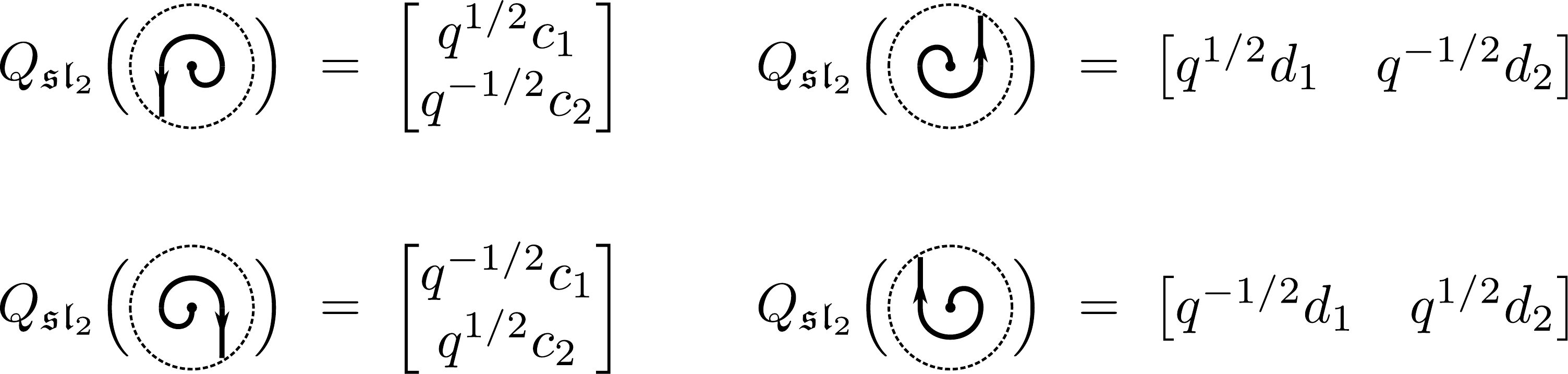}
    \caption{Value of $Q_{\mathfrak{sl}_2}$ on coframing loops.}
    \label{fig:sl_2_coframing}
\end{figure}

If we take $c_1=d_1$ and $c_2=d_2$ then adding a $(+1)$-coframing loop clearly has the same effect when added around either end-point, and similarly for $(-1)$-coframing loops. Thus in this case $Q_{\mathfrak{sl}_2}$ is seen to be invariant under the coframing identities.

In particular taking $c_1=c_2=d_1=d_2=1/\sqrt{2}$ we conclude:
\begin{proposition}
There exists an invariant of \textit{biframed} planar knotoids, denoted $\overline{Q}_{\mathfrak{sl}_2}$, possessing the properties from Figure \ref{fig:sl_2_triframed_invariant} and satisfying
\[
    \overline{Q}_{\mathfrak{sl}_2}\Big(
    \raisebox{-.45cm}{\includegraphics[width=1.1cm]{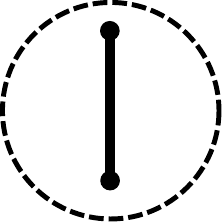}}
    \Big)
    = 1.
\]
\end{proposition}

One can adjust this invariant using the integer $\text{fr}(K)$ to obtain an $R1$-invariant version of $\overline{Q}_{\mathfrak{sl}_2}$ exactly as one does for knots. For knots this results in the Jones polynomial \cite[~Ch.~9.4]{jackson}. Using the integer $\text{cofr}(K)$ one can similarly adjust $\overline{Q}_{\mathfrak{sl}_2}$ to be coframing-invariant, resulting in an extension of the Jones polynomial to knotoids. Such an extension was already given in \cite{gugumcu2017new}, and these extensions must agree up to normalization since a knotoid invariant is uniquely defined by its Skein relations and its values on unknotted components (given that an invariant with these Skein relation exists, which is the statement of Proposition 3.11). Also note that the Kauffman bracket was already extended to knotoids in \cite{turaev2012knotoids}.

\end{example}

\subsubsection{Option 2}

Our second suggestion for extending $\varphi$ loses the advantage of being particularly simple, but in return has several pleasant features: it is guaranteed to be a biframed spherical knotoid invariant and requires no additional choices after $(\mathcal{C},V)$ are fixed. Moreover it has a strong topological motivation; namely we extend $\varphi$ by not just considering a biframed knotoid, but a planar projection of the entire theta-curve associated to it:

\begin{definition}\label{def:theta_invariant}
Let $B$ be a braided group in a rigid ribbon category $C$. Let $K$ be a biframed knotoid, assuming for simplicity that $K$ has integer coframing. Take a diagram of $K$ such that the tangent vectors of $K$ at the endpoints are vertically downwards. Now replace the leg of $K$ by a splitting $\Delta$ representing the coproduct of $B$, attaching $K$ on the left-hand side, and replace the head by a merging $\cdot$ representing the product of $B$, again attaching $K$ on the left. See figure \ref{fig:splitting_and_merging}.

\begin{figure}[ht]
    \centering
    \includegraphics[width=.85\linewidth]{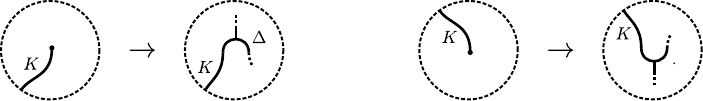}
    \caption{Attaching $\Delta$ and $\cdot$ to the leg and head of $K$, respectively. The dotted lines depict 4 open ends.}
    \label{fig:splitting_and_merging}
\end{figure}

The additions $\Delta,\cdot$ now have 4 open ends: two on the right-hand side and two vertically outgoing ends. Attach the right-hand ends via a line segment going under any strand of $K$ it encounters. Extend the outgoing ends vertically, going over all of $K$. Take the closure of the resulting decorated tangle to obtain a decorated tangle $\overline{K}$. Define $\varphi_\theta(K)$ to be the braided group extension of $\varphi$ evaluated on $\overline{K}$.
\end{definition}

\begin{example}
The construction of $\varphi_\theta(K)$ is depicted for an example knotoid in Figure \ref{fig:true_quantum_knotoid}.
\end{example}

\begin{figure}[ht]
    \centering
    \includegraphics[width=.45\linewidth]{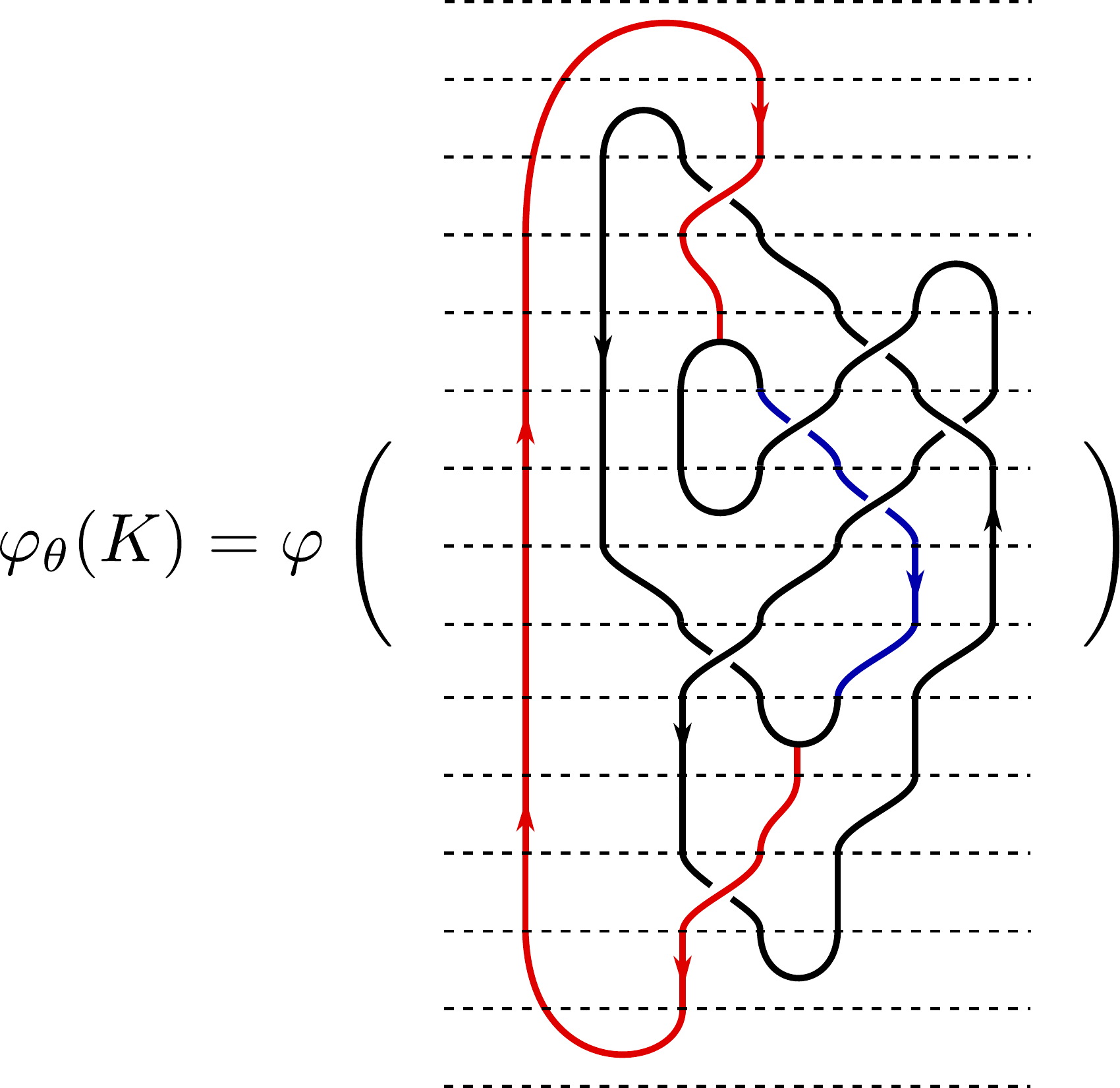}
    \caption{The construction of $\varphi_\theta(K)$. The knotoid $K$ is depicted in black.}
    \label{fig:true_quantum_knotoid}
\end{figure}

\begin{theorem}\label{thm:theta_invariant}
The map $\varphi_\theta(K)\in\Hom_\mathcal{C}(k,k)$ is a biframed knotoid invariant, but it is \textit{not} invariant under the forbidden moves from Figure \ref{fig:forbiddenmoves}.
\end{theorem}

\begin{proof}
Invariance under $R1',R2,R3$ and isotopies away from the end-points follows from the invariance of $\varphi$ under these moves, so for the first statement it suffices to prove invariance under the spherical move $R4$ and the coframing identities.

Invariance under $R4$ is clear, since strands can be moved over $(\cdot,\Delta)$ and hence all of $\overline{K}$, and oppositely oriented $R1$-loops cancel by Figure \ref{fig:cancellingloops}.

To show invariance under the first coframing identity from Figure \ref{fig:coframing_identities} we do some preliminary computations, depicted in Figures \ref{fig:coframing_inv_prereq2} and \ref{fig:coframing_inv_prereq1}.

\begin{figure}[ht]
    \centering
    \includegraphics[width=\linewidth]{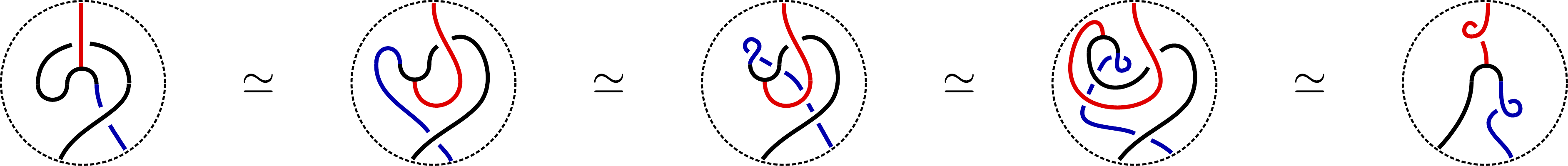}
    \caption{Turning coframing into framing at the leg of $K$.}
    \label{fig:coframing_inv_prereq2}
\end{figure}

\begin{figure}[ht]
    \centering
    \includegraphics[width=\linewidth]{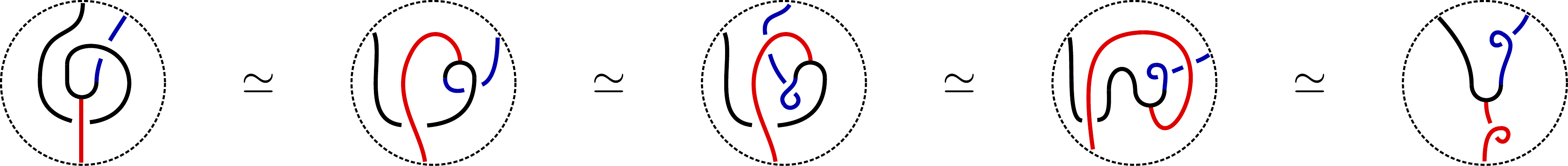}
    \caption{Turning coframing into framing at the head of $K$.}
    \label{fig:coframing_inv_prereq1}
\end{figure}

In Figures \ref{fig:coframing_inv_prereq2} and \ref{fig:coframing_inv_prereq1} we turn $(\cdot,\Delta)$ around to depict morphisms $(\Delta^*,\cdot^*)$ on $B^*$. The morphisms defined as such diagrammatically extend to a canonical braided group structure on $B^*$ for which the diagrammatic moves we have used hold by definition; see \cite[~Prop.~4.11]{majid1993beyond}.

Using these preliminaries, as well as two applications of $R1'$, invariance of $\varphi_\theta(K)$ follows from the sequence of moves depicted in Figure \ref{fig:true_coframing_inv}.

\begin{figure}[ht]
    \centering
    \includegraphics[width=\linewidth]{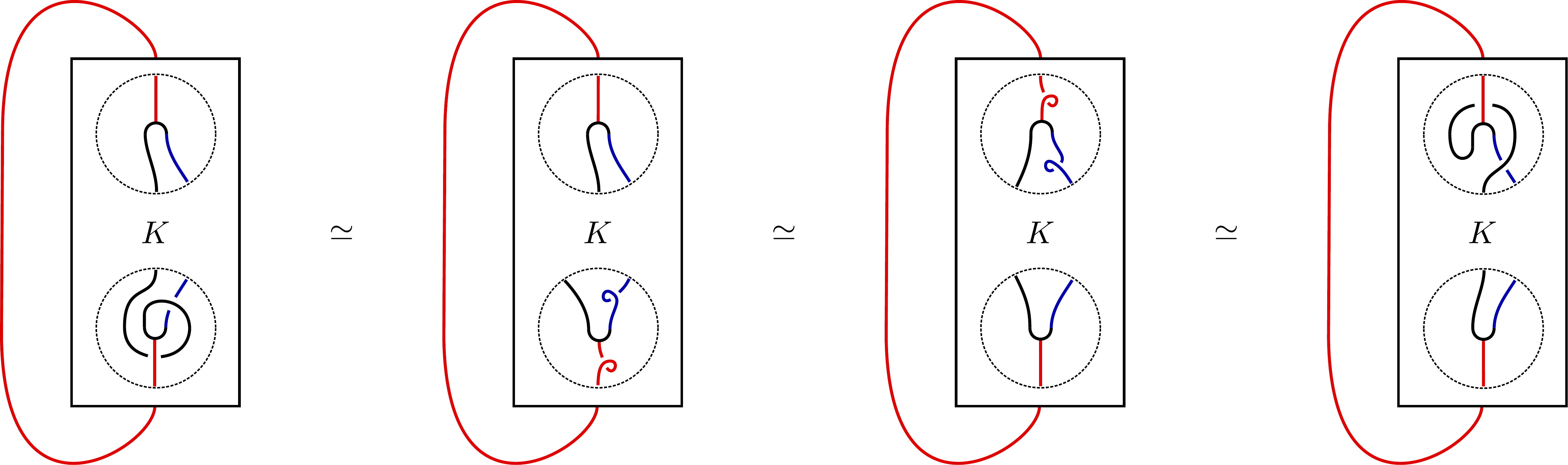}
    \caption{Invariance of $\overline{K}$ under the first coframing identity.}
    \label{fig:true_coframing_inv}
\end{figure}

Invariance under the other coframing identity is shown analogously. For the final statement of the theorem, let $D_1$ and $D_2$ denote the oriented tangle diagrams that $\varphi_\theta$ assigns to both sides of one of the moves in Figure \ref{fig:forbiddenmoves}. Then clearly $D_1$ and $D_2$ differ by at least one crossing change, and so in general $\varphi_\theta(K)$ cannot be invariant under either of these moves.
\end{proof}

\begin{remark}
The proof of Theorem \ref{thm:theta_invariant} shows in particular how $\varphi_\theta(K)$ changes under adjustment of the coframing: namely by adding an $R1$-loop to the strand of $\overline{K}$ representing $\overline{e}_+$, and an oppositely oriented $R1$-loop to the strand representing $\overline{e}_-$.

This remark also motivates why biframed knotoids are a natural object to consider if one wants to extend quantum knot invariants to knotoids or their associated theta-curves. Namely because quantum invariants are only $R1'$-invariant and not $R1$-invariant, one would expect the lines representing $\overline{e}_\pm$ to need to be framed as well. As the proof of Theorem \ref{thm:theta_invariant} shows, it turns out biframed knotoids provide a proper framework for this in the case of spherical knotoids.
% Note how coframing changes influence the invariant
% Note how nicely this proof shows the necessity of coframing for quantum invariants of spherical knotoids
\end{remark}

Although we have assumed $K$ has integer coframing in Definition \ref{def:theta_invariant}, the above generalizes readily to half-integer coframing. The only adjustment that needs to be made is that $\Delta$ must be replaced by $\cdot^*$ (or $\cdot$ by $\Delta^*$) at one endpoint. The construction is depicted (on its side) for an example half-integer coframing knotoid in Figure \ref{fig:quantum_halfcoframing}.

\begin{figure}[ht]
    \centering
    \includegraphics[width=.5\linewidth]{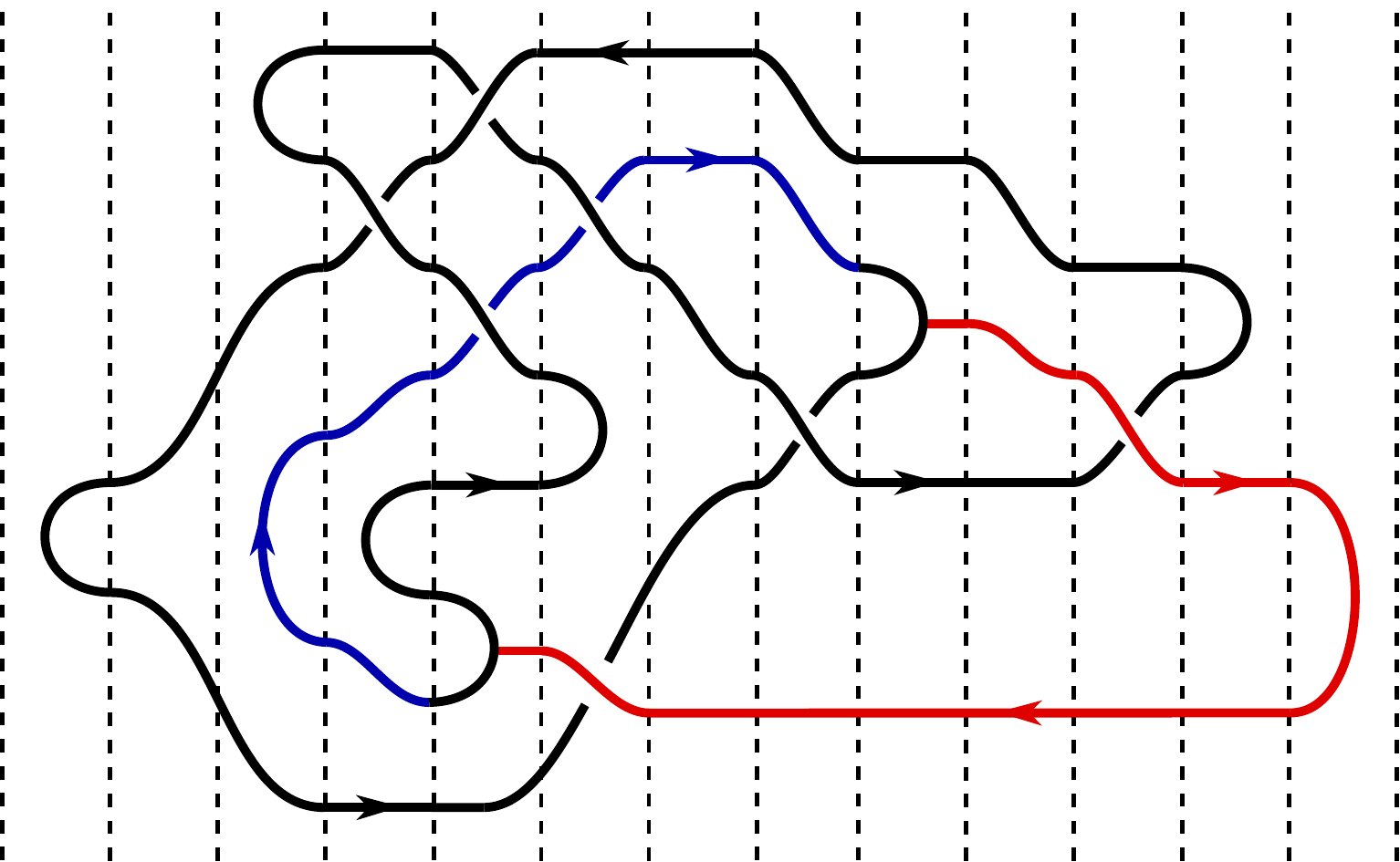}
    \caption{Construction of $\overline{K}$ for a half-integer-coframed knotoid (rotated $90^\circ$ counter-clockwise).}
    \label{fig:quantum_halfcoframing}
\end{figure}

As a closing remark, we briefly discuss the existence of the objects required to construct an invariant of the type suggested in Definition \ref{def:theta_invariant}. To see this existence we must exemplify braided groups in ribbon categories. Taking the ribbon category ${}_H\mathcal{M}_F$ for a Hopf algebra $H$ and noting that braided groups in ${}_H\mathcal{M}_F$ are Hopf algebra objects, it is natural to hope $H$ might be a braided group, or at least a bialgebra object in ${}_H\mathcal{M}_F$, when acting on itself suitably.

In fact if we take the adjoint action $\Ad$ of $H$ on itself, then $\cdot$ is automatically an intertwiner \cite[~Ch.~2]{majidprimer}. If moreover $H$ is cocommutative then $\Delta$ is also easily seen to be an intertwiner, giving an example of a nontrivial object suitable to apply Definition \ref{def:theta_invariant}.\\

\textbf{Acknowledgements}: My gratitude goes out to Agnese Barbensi, Ulrike Tillmann, Daniele Celoria, Jo Ellis-Monaghan, and Roland van der Veen for their advice and guidance.

% \noindent\textbf{Declarations}: 

% \begin{itemize}
%     \item Conflicts of interest: No funding was received to assist with the preparation of this manuscript. The author has no conflicts of interest to declare that are relevant to the content of this article.
%     \item Availability of data, material, and code: Not applicable.
%     % \item Code availability: Not applicable.
%     \item Ethics approval: Not applicable.
%     \item Consent to participate and for publication: Not applicable.
%     % \item Consent for publication: Not applicable.
% \end{itemize}

% \noindent Funding: No funding was received to assist with the preparation of this manuscript.

% \noindent Conflicts of interest: The author has no conflicts of interest to declare that are relevant to the content of this article.

% \noindent Availability of data and material: Not applicable.

% \noindent Code availability: Not applicable.

% \noindent Ethics approval: not applicable.

% \noindent Consent to participate: not applicable.

% \noindent Consent for publication: not applicable.

% \bibliography{sn-bibliography}% common bib file

\begin{thebibliography}{10}

\bibitem{bar2019polynomial}
Dror Bar-Natan and Roland van~der Veen, \emph{A polynomial time knot
  polynomial}, Proceedings of the American Mathematical Society \textbf{147}
  (2019), no.~1, 377--397.

\bibitem{barbensi2018double}
Agnese Barbensi, Dorothy Buck, Heather~A Harrington, and Marc Lackenby,
  \emph{Double branched covers of knotoids}, arXiv preprint arXiv:1811.09121
  (2018). To appear in Communications in Analysis and Geometry.

\bibitem{bates2005dna}
Andrew~D Bates, Anthony Maxwell, Tony Maxwell, et~al., \emph{Dna topology},
  Oxford University Press, USA, 2005.

\bibitem{chmutov2012introduction}
Sergei Chmutov, Sergei Duzhin, and Jacob Mostovoy, \emph{Introduction to
  vassiliev knot invariants}, Cambridge University Press, 2012.

\bibitem{elhamdadi2020framed}
Mohamed Elhamdadi, Mustafa Hajij, and Kyle Istvan, \emph{Framed knots}, The
  Mathematical Intelligencer \textbf{42} (2020), no.~4, 7--22.

\bibitem{goundaroulis2019systematic}
Dimos Goundaroulis, Julien Dorier, and Andrzej Stasiak, \emph{A systematic
  classification of knotoids on the plane and on the sphere}, arXiv preprint
  arXiv:1902.07277 (2019).

\bibitem{goundaroulis2020knotoids}
Dimos Goundaroulis, Julien Dorier, and Andrzej Stasiak, \emph{Knotoids and protein structure}, Topology and Geometry of
  Biopolymers \textbf{746} (2020), 185.

\bibitem{gugumcu2017new}
Neslihan G{\"u}g{\"u}mc{\"u} and Louis~H Kauffman, \emph{New invariants of
  knotoids}, European Journal of Combinatorics \textbf{65} (2017), 186--229.

\bibitem{gugumcu2021quantum}
Neslihan G{\"u}g{\"u}mc{\"u} and Louis~H Kauffman, \emph{Quantum invariants of knotoids}, Communications in Mathematical
  Physics (2021), 1--48.

\bibitem{jackson}
David~M Jackson and Iain Moffatt, \emph{An introduction to quantum and
  vassiliev knot invariants}, Springer, 2019.

\bibitem{majid1993beyond}
Shahn Majid, \emph{Beyond supersymmetry and quantum symmetry (an introduction
  to braided-groups and braided-matrices)}, Quantum Groups, Integrable
  Statistical Models and Knot Theory (1993), 231--282.

\bibitem{majidfoundations}
Shahn Majid, \emph{Foundations of quantum group theory}, Cambridge university
  press, 2000.

\bibitem{majidprimer}
Shahn Majid, \emph{A quantum groups primer}, vol. 292, Cambridge University Press,
  2002.

\bibitem{rolfsen}
Dale Rolfsen, \emph{Knots and links}, vol. 346, American Mathematical Soc.,
  2003.

\bibitem{turaev2012knotoids}
Vladimir Turaev, \emph{Knotoids}, Osaka Journal of Mathematics \textbf{49}
  (2012), no.~1, 195--223.

\bibitem{vanderveen}
Roland van~der Veen, \emph{Introduction to quantum invariants of knots}, 2016
  MATRIX Annals, Springer, 2018, pp.~637--656.

\end{thebibliography}


\end{document}